\documentclass[11pt,reqno]{amsart}

\usepackage[T1]{fontenc}
\usepackage{lmodern}
\usepackage{microtype}
\usepackage{xcolor}
\usepackage[margin=0.93in]{geometry}
\usepackage{amsmath,amssymb,mathtools,mathrsfs}
\usepackage{enumitem}
\usepackage[hidelinks]{hyperref}
\usepackage[nameinlink,capitalise,noabbrev]{cleveref}

\allowdisplaybreaks

\newtheorem{theorem}{Theorem}[section]
\newtheorem{proposition}[theorem]{Proposition}
\newtheorem{fact}[theorem]{Fact}
\newtheorem{lemma}[theorem]{Lemma}
\newtheorem{corollary}[theorem]{Corollary}
\newtheorem{claim}[theorem]{Claim}
\theoremstyle{definition}

\newtheorem{remark}[theorem]{Remark}

\numberwithin{equation}{section}

\newcommand{\K}{\mathcal K}

\newcommand{\B}{\mathbb B}
\newcommand{\one}{\mathbf 1}
\newcommand{\supp}{\operatorname{supp}}
\newcommand{\dist}{\operatorname{dist}}
\newcommand{\diam}{\operatorname{diam}}
\newcommand{\NCap}{\operatorname{Cap}}
\newcommand{\LE}{\operatorname{LE}}
\newcommand{\tr}{\operatorname{tr}}

\newcommand{\dd}{\,\mathrm d}
\newcommand{\eps}{\varepsilon}
\newcommand{\Z}{\mathbb Z}
\newcommand{\Q}{\mathbb Q}
\newcommand{\R}{\mathbb R}
\newcommand{\Hd}{d_{\mathrm H}}
\newcommand{\cH}{\mathcal H}

\crefname{claim}{Claim}{Claims}
\Crefname{claim}{Claim}{Claims}

\title[Three-dimensional LERW and loop-catchers]
{Characterization and scaling limit of three-dimensional \\loop-erased random walk
and random-walk loop-catchers}
\author{Gefei Cai}
\address{Beijing International Center for Mathematical Research,
Peking University}
\email{caigefei1917@pku.edu.cn}

\author{Xinyi Li}
\address{Beijing International Center for Mathematical Research,
Peking University}
\email{xinyili@bicmr.pku.edu.cn}

\author{Daisuke Shiraishi}
\address{Graduate School of Informatics,
Kyoto University}
\email{shiraishi@acs.i.kyoto-u.ac.jp}

\date{\today}
\date{\today}

\begin{document}

\begin{abstract}

Random-walk loop-catchers (RWLC) are a one-parameter family of random subsets of the trace of a random walk, recently introduced in \cite{Cai26}, which interpolates between loop-erased random walk (LERW) and the full walk trace, satisfying a recovery property.
We prove that the continuous version of this recovery property uniquely characterizes the scaling limit of LERW and RWLC in three dimensions, using a Green-function test via entangled multipath LERW, as outlined in \cite{Cai26}. 
For LERW, this resolves the 3D case of the conjecture of Sapozhnikov--Shiraishi \cite{SS18} and gives a new and axiomatic proof of the existence of the scaling limit, which was first proved in \cite{Kozma}. For RWLC, this gives the construction of 3D Brownian loop-catchers. The characterization also yields inversion invariance of the 3D LERW scaling limit.

\end{abstract}

\maketitle
\enlargethispage{3pt}

\setcounter{tocdepth}{1}
\tableofcontents

\section{Introduction and main results}\label{sec:intro}

\subsection*{Background and motivation}
Loop-erased random walk (LERW), introduced by Lawler~\cite{Lawler80},
is obtained by chronologically erasing the loops of a random walk.
LERW also describes branches of uniform spanning trees, and
Wilson's algorithm~\cite{Wilson96} constructs the tree by successively
adding loop-erased walks. 

In two dimensions, Schramm~\cite{Schramm00} showed how conformal
invariance and the domain Markov property lead to a Loewner description
of the limiting curve. Lawler--Schramm--Werner~\cite{LSW04} subsequently
proved convergence of planar LERW to \(\mathrm{SLE}_2\). In dimensions \(d\ge4\), the scaling limit is Brownian motion,
with a logarithmic correction to the time scaling in dimension four
\cite{Lawler80,Lawler86,Lawler95}.
In three dimensions, Kozma~\cite{Kozma} constructed the scaling limit
as a random compact set and established its invariance under rotations
and dilations.
Shiraishi~\cite{ShiraishiGrowth} proved the existence of the
growth exponent of three-dimensional LERW.
Li--Shiraishi~\cite{LiShiraishi} later proved convergence
in the natural parametrization, with renormalized discrete length
converging to a measure on the limiting curve.
These results establish the continuum object and its natural time
scale. A further question is how to characterize its law using
continuum data alone.

Brownian motion is a natural source of such data, since the walk before
loop erasure converges to Brownian motion. However, convergence of the
underlying paths does not by itself pass chronological loop erasure to
the limit: erasure depends on exact intersections, which are unstable
under uniform approximation. 
Brownian loop soup, introduced in the plane by
Lawler--Werner~\cite{LawlerWerner}, provides a collection of loops that
can be sampled independently of a candidate curve. Its discrete
counterpart relates a random walk to its loop erasure through independent
loop addition; see \cite{Lawler18}.

Sapozhnikov--Shiraishi~\cite{SS18} made this connection precise in
dimensions two and three. In the radial unit-ball setting, they proved
that every subsequential LERW limit is a simple curve and that its union
with the loops of an independent intensity-one Brownian loop soup
intersecting it has, after closure, the law of a stopped Brownian range.
Their Conjecture~1.3 asks whether this recovery property uniquely
determines the curve law. A positive answer gives an intrinsic
description of continuum LERW in terms of Brownian motion and Brownian
loop measure. It also separates the identification of a subsequential
limit from the discrete estimates needed to construct it.

Cai~\cite{Cai26} introduced random-walk loop-catchers (RWLC) on finite graphs
and constructed their planar Brownian limits. With the parameter written
as the soup intensity \(\lambda\in[0,1]\), the finite models range from
the random-walk trace at \(\lambda=0\) to LERW at \(\lambda=1\).
In that paper, the finite recovery property (see \eqref{eq:blc-discrete-recovery-transform} in this work) uses a single independent soup whose loops
meeting the catcher restore the original walk range in law.
The entangled multipath LERW construction in that work yields a Green-function test: random probes convert recovery identities such as \eqref{eq:recovery-transform} into mixed
moments of killed Green functions. This proves uniqueness of the full
planar loop-catcher law, and in particular implies the two-dimensional version of~\cite[Conjecture 1.3]{SS18}.

We carry out the three-dimensional extension outlined in
\cite[Section~1.5]{Cai26}, proving the required hittability estimates
and passing the finite recovery identities
\eqref{eq:blc-discrete-recovery-transform} and
\eqref{eq:blc-discrete-probe-transform} to the limit.
In particular, we resolve Conjecture~1.3 of
Sapozhnikov--Shiraishi~\cite{SS18} in three dimensions and construct
the three-dimensional Brownian loop-catchers as the scaling limit of RWLC.

\subsection*{Main results}

Fix a bounded connected Lipschitz domain \(B\subset\R^3\)
containing \(0\). 

Let \(\Gamma\) be the space of
{\it unparametrized curves} \(K\subset\overline B\) from \(0\) to
\(\partial B\), with no other point on \(\partial B\). We give \(\Gamma\)
the Borel structure inherited from the Hausdorff metric on the space of nonempty compact
subsets of \(\overline B\). For \(\lambda\ge0\), let
\(\mathcal L_B^\lambda\) be a Brownian loop soup of intensity \(\lambda\),
independent of \(K\), and define its direct loop decoration by
\begin{equation}\label{eq:decoration}
 \mathcal D_\lambda(K)
 :=\overline{K\cup\bigcup_{\ell\in\mathcal L_B^\lambda:\,
                           \ell\cap K\ne\varnothing}\ell}.
\end{equation}
The closure is taken in \(\overline B\), and the soup is empty at
\(\lambda=0\). A probability measure \(\mathfrak m\) on \(\Gamma\)
satisfies the \emph{intensity-one recovery property} if, for \(K\sim\mathfrak m\),
\begin{equation}\label{eq:lerw-recovery}
 \mathcal D_1(K)\stackrel d=W[0,\tau_B],
\end{equation}
where \(W\) is Brownian motion started at \(0\) and \(\tau_B\) is its
first exit time from \(B\).

\begin{theorem}
\label{thm:main}
There is a unique probability measure \(\mathfrak m\) on \(\Gamma\) that
satisfies the (intensity-one) recovery property.  It is the scaling-limit law, in the
unparametrized Hausdorff trace topology, of three-dimensional loop-erased
random walk from \(0\) to \(\partial B\).
\end{theorem}

The characterization also gives another route to full Hausdorff
convergence in \(B\). Compactness supplies subsequential limits,
and the recovery property in \cite[Theorems~1.1--1.2]{SS18}
extend to Lipschitz domains as explained in \cref{sec:domain-extension};
\cref{thm:main} then identifies their common law.

\medskip

We next state the second main result, the construction of three-dimensional Brownian loop-catchers. 
For every $\lambda\in[0,1]$, \cite[Theorem~1.3]{Cai26} associates with the random walk in
\(D^\delta=B\cap\delta\mathbb Z^3\) started from the origin and killed on leaving \(D^\delta\) a random vertex set
\(S_\lambda^\delta\), the \(\lambda\)-random-walk loop-catcher.  Its cell
trace is denoted by
\(K_\lambda^\delta\); see \cref{sec:discrete-model} for the precise 
convention. 
The construction in \cite[Algorithm~2.5 and Definition~2.6]{Cai26}
also supplies a time-parametrized directed path with range
\(S_\lambda^\delta\). 

Let \(\mathfrak C_B\) be the Borel space of compact connected
\(K\subset\overline B\) such that \(0\in K\) and \(K\cap\partial B\) is a
singleton.
For compact sets \(A,K\subset\overline B\), let \(m_B(K,A)\) stand for the measure of the loops in $B$ intersecting both $K$ and $A$, as defined in \eqref{eq:loop-mes}.
\begin{theorem}
\label{thm:blc-main}
For each fixed \(0<\lambda\leq1\), the laws of \(K_\lambda^\delta\)
converge as \(\delta=1/N\downarrow0\), in the Hausdorff topology, to a law
\(\nu_\lambda\) on \(\mathfrak C_B\).  If
\(K_\lambda^B\sim\nu_\lambda\), then \(K_\lambda^B\) has empty interior and,
for every compact \(A\Subset B\),
\begin{equation}\label{eq:blc-recovery-intro}
 \mathbb E\!\left[
  \one_{\{K_\lambda^B\cap A=\varnothing\}}
  e^{-\lambda m_B(K_\lambda^B,A)}\right]
 =\mathbb P^0\!\left(W[0,\tau_B]\cap A=\varnothing\right).
\end{equation}
Equivalently,
\begin{equation}\label{eq:looprecovery}
 \mathcal D_\lambda(K_\lambda^B)
 \stackrel d= W[0,\tau_B].
\end{equation}
Moreover, \(\nu_\lambda\) is the unique probability law on
\(\mathfrak C_B\) satisfying \eqref{eq:blc-recovery-intro}.  For \(\lambda=1\), $\nu_1$ reduces to \(\mathfrak m\) in \cref{thm:main}.
\end{theorem}
Let us remark that for loop-catchers, convergence as curves modulo
increasing reparametrization is a stronger assertion, which is not yet
established by our argument: the recovery property \eqref{eq:looprecovery} determines the range
but does not keep its order of traversal.

We now briefly comment on the scope and consequences of our results.

\begin{remark}[Loop-erased Brownian motion]
\Cref{thm:main} gives an intrinsic justification for calling the
three-dimensional LERW scaling limit \emph{loop-erased Brownian motion} (LEBM):
its law is uniquely characterized by the Brownian recovery property.
This is uniqueness in distribution. Whether a Brownian trajectory almost
surely has a unique measurable loop-erasure, up to increasing reparametrization,
remains open in both two and three dimensions. 
\end{remark}

The characterization also gives inversion invariance of the
three-dimensional LERW scaling limit, which would be cumbersome to obtain
using the framework of \cite{Kozma}. We state the whole-space version
below. The bounded-domain statement, including its endpoint conventions,
is given in \Cref{rem:bdeddomain}: it covers normalized bridges between distinct
interior points and the interior-to-boundary exit-excursion formulation
of Section~\ref{sec:discrete-model}, with the exit law transformed under inversion.
\begin{corollary}\label{cor:inversion}
Let \(K_\infty\) be the trace of the scaling limit of infinite
three-dimensional LERW from \(0\), as in
\cite[Theorem~10.1]{LiShiraishi}, with its endpoint at infinity adjoined.
Regard this as a random compact subset of
\(\widehat{\R^3}=\R^3\cup\{\infty\}\), and extend
\(\iota(x)=x/|x|^2\) by \(\iota(0)=\infty\) and \(\iota(\infty)=0\).
Then
\[
 \iota(K_\infty)\stackrel{\mathrm d}=K_\infty.
\]
For oriented curves from \(0\) to \(\infty\), inversion is accompanied
by reversal of orientation. The assertion is modulo reparametrization.
\end{corollary}
We will give a proof sketch of \Cref{cor:inversion} in Appendix~\ref{sec:whole-space}. 

\begin{remark}[Continuum Laplacian-\(1\) motion]
The problem of defining Laplacian-\(b\) motion in three dimensions
was already formulated by Lawler~\cite[Section~1]{LawlerLaplacian}.
At \(b=1\), it would be interesting to characterize the loop-erased
Brownian motion constructed here by this local growth rule. In the discrete model, conditional on an initial
segment \(\eta\), the next-step weights are proportional to
\(p(\eta_{\mathrm{tip}},y)h_\eta(y)\), where \(h_\eta(y)\) is the
probability of reaching the target before hitting \(\eta\).
In a continuum domain \(D\), the corresponding function is
\[
 h_t(x)=\mathbb P^x\bigl(\tau_{\partial D}
                         <\tau_{\gamma[0,t]}\bigr).
\]
The problem is to define the normalized first-power harmonic weights
at the random tip, prove independence from the tip regularization,
and establish existence and uniqueness of the resulting growth law.
One possible route to identification is to derive the intensity-one
recovery property \eqref{eq:lerw-recovery} from this local rule and then apply
\cref{thm:main}. Let us remark that the comparison of harmonic functions in the proof of the boundary Harnack principle \cite[Theorem~1.5 and Remark~1.3]{GLLLL24} may help control the normalized
weights near the tip.
 This question is left for our future work.
\end{remark}

\begin{remark}[Universality of the three-dimensional uniform spanning tree]
\label{rem:universality}
The characterization proved here also provides a route to universality.
In forthcoming work with D.~A. Croydon \cite{CCLSprep}, we study
universality of the scaling limit of the three-dimensional
uniform spanning tree in Schramm's topology.
For earlier results on scaling limits of the three-dimensional
uniform spanning tree, see \cite{ACHS21}.
\end{remark}

\begin{remark}[Convergence rates]\label{rem:rates}
This work does not establish a power-law rate in the L\'evy--Prokhorov metric on laws of compact
sets, equipped with Hausdorff distance. 
However, for the sharp one-point estimates of
Li--Shiraishi~\cite{LiShiraishiOnePoint}, the needed convergence input
is more specific: a polynomial comparison rate for the macroscopic
non-intersection probabilities of simple random walk (SRW) and LERW, to which the present framework
indeed offers a route through quantitative refinements of the
Green-function, bridge-convergence, hittability, and near-miss
estimates; we leave these refinements to curious readers.
\end{remark}

\subsection*{Proof architecture}

The proof follows the outline in~\cite[Section 1.5]{Cai26} by implementing the Green-function test of \cite{Cai26} to three-dimensions.
According to Fact \ref{lem:closed-decoration-avoidance}, the recovery property
gives the weighted avoidance identities
\begin{equation}\label{eq:recovery-transform}
 F_{\mathfrak m}(A):=
 \mathbb E_{\mathfrak m}\!\left[
  \one_{\{K\cap A=\varnothing\}}e^{-m_B(K,A)}\right]
 =\mathbb P^0\big(W[0,\tau_B]\cap A=\varnothing\big).
\end{equation}
We recover the law of \(K\) from this transform in three steps.
\begin{enumerate}[label=\textnormal{(\roman*)},leftmargin=2.4em]
\item \emph{Finite probes and their limits.}
\Cref{prop:finite-probe} recalls Cai's recovery property \eqref{eq:finite-transform} for entangled multipath LERW. Bridge convergence and the common decoration-convergence
result in \cref{lem:green-path-limit,prop:common-decoration-stability,cor:recovery-transfer} yield
the continuum probe identity \cref{prop:continuum-probe}.
The key hittability estimate for 3D entangled
multipath LERW is stated in
\cref{prop:multipath-hittability} and proved in \cref{sec:hittability-input}.

\item \emph{Uniqueness from Green functions.}
The recovery law \eqref{eq:lerw-recovery} forces local nonpolarity. An independent probe converts
the weighted transform into all mixed moments of normalized Green functions. The common results
\cref{lem:no-polar-subarc,prop:green-injective,lem:mixed-moments,rem:borel-inverse}
in \cref{sec:prelim} show that these coordinates determine the law of the random compact set. Applying them to the intensity-one probes proves \cref{thm:main}.

\item \emph{General loop-catchers.}
For every catcher vertex, a recovered random walk has an initial segment
whose loop erasure ends at that vertex and is contained in the catcher
(\cref{lem:blc-online-witness}). This inclusion transfers the hittability
estimate to every intensity in \cref{prop:blc-uniform-hittability}.
The same decoration-convergence result then gives
\cref{prop:blc-continuum-identities}, and the same Green-function
identification proves \cref{thm:blc-main}. The shared arguments are collected in \cref{sec:prelim,sec:decoration}.
\end{enumerate}

\noindent{\bf Acknowledgements}: During D.S.'s stay at ETH Zurich in 2014--2016, Alain-Sol Sznitman suggested approaching \cite[Conjecture~1.3]{SS18} by expressing $\mathbb{E}[f(K)]$, for a sufficiently rich class of functionals $f$, solely in terms of Brownian motion, thereby characterizing the law of $K$. D.S. is grateful to him for this insightful and far-sighted suggestion. G.C.\ and X.L.\ are supported by the National Key R\&D Program of China (No.~2021YFA1002700) and Beijing Natural Science Foundation (JQ26001). D.S.\ was supported in part by JSPS KAKENHI Grant Numbers JP26K00609 and JP23K22399.

AI tools were only used for language polishing and proofreading of the manuscript. The authors take full responsibility for its content.

\section{Preliminaries}\label{sec:prelim}

This section fixes the discrete and continuum conventions, recalls
Cai's finite recovery identities
\eqref{eq:blc-discrete-recovery-transform}--\eqref{eq:blc-discrete-probe-transform},
and collects the Green-function identification tools and bridge
convergence needed below.

\subsection{Notation and conventions}

All distances and balls are Euclidean unless a lattice subscript is
shown. For a set \(A\), write \(\overline A\), \(A^c\),
\(\partial A\), and \(\operatorname{int}A\) for its closure,
complement, boundary, and interior. We write \(A\Subset D\) when
\(\overline A\) is compact and contained in the open set \(D\).
For nonempty sets \(A,C\),
\[
 \dist(A,C)=\inf_{x\in A,y\in C}|x-y|,
 \qquad \diam(A)=\sup_{x,y\in A}|x-y|.
\]
The Hausdorff distance between nonempty compact sets is
\[
 \Hd(A,C)=\max\!\left\{
   \sup_{x\in A}\dist(x,C),\,
   \sup_{y\in C}\dist(y,A)\right\}.
\]
We distinguish an ordered discrete path \(\omega\), its vertex support
\(\supp(\omega)\), and the trace \(\tr(\gamma)\) of a continuous or
linearly interpolated path. Loop masses use vertex supports in the
discrete setting; Hausdorff convergence concerns the associated compact
traces. The notation \(\mathbf1_E\) denotes the indicator of an event
\(E\), \(\stackrel d=\) equality in distribution, and \(\Rightarrow\)
weak convergence in the stated topology. Unlabelled probabilities and
expectations refer to all random objects in the expression; a superscript
\(x\) specifies the starting point of a walk or Brownian motion.

Positive constants \(c,C\) may change from line to line. Their dependence
on fixed domains, compact sets, and finite endpoint lists is allowed
unless stated otherwise; they are independent of the mesh and of the
small parameters in the displayed limits. Subscripts indicate any other
allowed dependence. We use \(f\lesssim g\) for \(f\le Cg\),
\(f\asymp g\) for two-sided comparison, \(f=O(g)\) for
\(|f|\le C|g|\), and \(f=o(g)\) for \(f/g\to0\), in the specified
limit.

\subsection{Brownian loops and normalized Brownian bridges}
We fix the Brownian loop and bridge normalizations and express loop decoration through avoidance probabilities.

For a bounded open \(D\subset\R^3\), Brownian motion has generator
\(\tfrac12\Delta\). Let \(p_D(t,x,y)\) be its killed heat kernel and

\begin{equation*}
 G_D(x,y)=\int_0^\infty p_D(t,x,y)\dd t.
\end{equation*}
For \(t>0\) and \(x,y\in D\) in the same component, let
\(\mathbf P_{x,y}^{D,t}\) be the killed Brownian bridge of duration \(t\).
For distinct \(x,y\), following \cite[(3.2)]{Cai26}, the normalized
Brownian bridge \(\Gamma_D^{x,y}\) is defined by
\begin{equation}\label{eq:green-bridge}
 \mathbb P(\Gamma_D^{x,y}\in\,\cdot\,)
 =\frac1{G_D(x,y)}\int_0^\infty
       p_D(t,x,y)\mathbf P_{x,y}^{D,t}(\,\cdot\,)\dd t.
\end{equation}

Its sample has a random finite duration, joins \(x\) to \(y\), and
is conditioned to remain in \(D\) throughout that duration.

We use the unrooted Brownian loop measure \(\mu_D^{\mathrm{loop}}\) with the conventional
normalization.  More precisely, the rooted measure on the right-hand side of
\begin{equation}\label{eq:brownian-loop-measure}
 \widetilde\mu_D^{\mathrm{root}}(\cdot)
 =\int_D\int_0^\infty
       \frac1t p_D(t,z,z)\mathbf P_{z,z}^{D,t}(\cdot)\dd t\dd z.
\end{equation}
is pushed forward under the quotient by cyclic time shift; this is the
definition of \(\mu_D^{\mathrm{loop}}\) used below.  Equivalently, it is the unrooted
push-forward of the rooted convention in \cite[Section~2.3]{SS18},
restricted to loops contained in \(D\); see also \cite{LawlerWerner}.
We henceforth identify a loop with its trace whenever intersections or
Hausdorff distance are considered.  For compact \(A,C\), define
\begin{equation}\label{eq:loop-mes}
 m_D(A,C):=\mu_D^{\mathrm{loop}}\{\ell:\ell\cap A\ne\varnothing,
                            \ \ell\cap C\ne\varnothing\}.
\end{equation}

For compact \(K\subset\overline D\), we also use
\(\mathcal D_\lambda(K)\) for \eqref{eq:decoration} with \(B\) replaced
by \(D\); the ambient domain will be clear from context.

The mutual loop mass is symmetric and may be infinite when the sets meet.
Every kernel
\(\one_{\{A\cap C=\varnothing\}}e^{-\lambda m_D(A,C)}\)
is defined to be zero on the intersecting event and to have the displayed
exponential value on the disjoint event, where the mass is finite.
We use the Newtonian kernel \(|x-y|^{-1}\) to define capacity.

We repeatedly use two elementary cutoff facts.  If compact
\(A,C\subset\overline D\) have distance at least \(d>0\), then
\begin{equation}\label{eq:separated-loop-finite}
 m_D(A,C)<\infty.
\end{equation}
Indeed, for duration \(t<1\), a loop meeting both sets makes a displacement
at least \(d\), and the Brownian-bridge Gaussian bound is
\(O(e^{-c d^2/t})\); for \(t\ge1\), the killed heat trace decays exponentially by the spectral gap.  The same proof, uniformly on lattice
approximations, will be used without further comment.  Second, for fixed
\(a>0\), the mass of loops with \(\diam(\ell)\ge a\) is locally finite.

Note that closure does not affect avoidance of a fixed compact set, allowing the Poisson formula to compute the avoidance probability of the decoration.
\begin{fact}
\label{lem:closed-decoration-avoidance}
Let \(D\subset\R^3\) be bounded and open, let
\(K,A\subset\overline D\) be compact, and let
\(\mathcal L_D^\lambda\) be a Brownian loop soup of intensity
\(0<\lambda<\infty\). Let \(H\) be the union in \eqref{eq:decoration} before taking closure, so that \(\overline H=\mathcal D_\lambda(K)\):
\[
 H=K\cup\bigcup_{\ell\in\mathcal L_D^\lambda:
                         \,\ell\cap K\ne\varnothing}\ell.
\]
Almost surely, \(\overline H\cap A=\varnothing\) if and only if
\(H\cap A=\varnothing\).  Consequently, for a random compact \(K\)
independent of the soup and every deterministic compact \(A\),
\[
 \mathbb P(\overline H\cap A=\varnothing\mid K)
 =\one_{\{K\cap A=\varnothing\}}e^{-\lambda m_D(K,A)}.
\]
Here the closure is taken in \(\overline D\).
\end{fact}

\subsection{Discrete conventions}\label{sec:discrete-conventions}We record the discrete counterparts of these objects and the estimates used to pass to the continuum.
We take \(\delta=N^{-1}\), \(N\in\mathbb N\), and let
\(D^\delta=D\cap\delta\Z^3\); all marked lattice approximations lie
in the same lattice component for every sufficiently small \(\delta\).
In the notation of \cite[Section~3.2]{Cai26}, write \(Q^\delta\) for the
killed transition matrix,
\begin{equation*}
 Q^\delta_{uv}=\frac16
 \one_{\{u,v\in D^\delta,\ |u-v|=\delta\}},
\end{equation*}
and let \(Q^{\delta,n}(u,v)\) be its \(n\)-step kernel.  Thus simple
random walk is killed on leaving \(D^\delta\), and
\begin{equation*}
 \mathcal G_{D^\delta}(u,v)=\sum_{n\ge0}Q^{\delta,n}(u,v)
 =\bigl(I-Q^\delta\bigr)^{-1}_{uv}.
\end{equation*}
A normalized random-walk bridge from \(a\) to \(b\) assigns a finite
nearest-neighbour path \(\omega=(\omega_0,\ldots,\omega_n)\subset D^\delta\)
from \(a\) to \(b\) the mass
\begin{equation*}
 \mathbb P_{D^\delta}^{a,b}(\omega)
 =\frac{q^\delta(\omega)}{\mathcal G_{D^\delta}(a,b)},\qquad
 q^\delta(\omega)=\prod_{j=1}^n Q^\delta_{\omega_{j-1}\omega_j}.
\end{equation*}

Time is rescaled by \(\delta^2/3\), so that linearly interpolated walk
converges to standard Brownian motion. Fix a compact box \(J_D\) whose
interior contains a neighbourhood of \(\overline D\). Interpolated lattice
paths and their traces are viewed in \(J_D\): for nonconvex \(D\), an
interpolated edge can leave \(\overline D\), but stays within \(O(\delta)\)
of its vertex support. Limit traces therefore lie in \(\overline D\).
For a path of lifetime \(t\), the variable-duration uniform topology
records \((t,\gamma(t\,\cdot))\) in
\((0,\infty)\times C([0,1],J_D)\).

For a set \(F\subset\R^3\) and \(\rho>0\), put
\begin{equation*}
 F^{+\rho}=\{x\in\R^3:\dist(x,F)\le\rho\}.
\end{equation*}

This is the closed Euclidean neighbourhood; after restriction to the
domain it corresponds to the notation \(F^{(\rho)}\) in
\cite[Section~3.3]{Cai26}. We keep the superscript \(+\rho\) to allow
neighbourhoods in the ambient box \(J_D\), including points outside \(D\).

We use the lattice balls and exit times
\begin{equation}\label{eq:lattice-ball-exit}
 \mathbb B_\delta(x,r)=\{y\in\delta\Z^3:|y-x|\le r\},\qquad
 \tau_{x,r}=\inf\{n\ge0:S_n\notin\mathbb B_\delta(x,r)\}.
\end{equation}
Write \(\tau_{D^\delta}=\inf\{n\ge0:S_n\notin D^\delta\}\), and let
\(p_n^\delta(u,v)\) denote the unrestricted \(n\)-step transition kernel on
\(\delta\Z^3\).

The unrooted random-walk loop
measure \(m_{D^\delta}^{\delta}\) is the push-forward of

\begin{equation*}
 \sum_{z\in D^\delta}\sum_{n\ge1}\frac1n
 Q^{\delta,n}(z,z)\mathbf P_{z,z}^{\delta,n}
\end{equation*}
under the quotient of cyclic shift; here
\(\mathbf P_{z,z}^{\delta,n}\) is the \(n\)-step walk bridge conditioned
to return to \(z\).  This normalization is compatible with
\eqref{eq:brownian-loop-measure} after diffusive scaling.
For vertex sets \(A,C\subset D^\delta\), we write
\begin{equation*}
 m_{D^\delta}^{\delta}(A,C)
 =m_{D^\delta}^{\delta}
   \{\ell:\ell\cap A\ne\varnothing,\ \ell\cap C\ne\varnothing\}
\end{equation*}
for the corresponding mutual random-walk loop mass.  Thus mutual masses
always use vertex supports; cell traces and linear interpolations are used
only for geometric convergence.

We collect the standard random-walk estimates needed below, with their three-dimensional normalization. Whenever
\(U\Subset D\), \(u,v\in U\cap\delta\Z^3\), and \(u\ne v\),
\begin{equation}\label{eq:green-3d-bounds}
 \mathcal G_{D^\delta}(u,v)\le C\frac{\delta}{|u-v|},
 \qquad
 \mathcal G_{D^\delta}(u,v)\ge c\delta
\end{equation}
in the second inequality when \(u,v\) approximate two fixed distinct
points in the same component. The upper bound is
\cite[Theorem~4.3.1]{LawlerLimic}; the lower bound follows from
\eqref{eq:green-normalization} below, with a constant depending on the
limiting endpoints and \(D\). We also use the Gaussian heat-kernel bound
\begin{equation}\label{eq:lclt}
 Q^{\delta,n}(u,v)\le Cn^{-3/2}
       \exp\{-c|u-v|^2/(n\delta^2)\},\qquad n\ge1,
\end{equation}
the bounded-domain spectral estimate
\begin{equation}\label{eq:discrete-spectral-bound}
 Q^{\delta,n}(u,v)
 \le C\delta^3e^{-c n\delta^2},
 \qquad n\ge\delta^{-2},
\end{equation}
and the exit estimate
\begin{equation*}
 \sup_u\mathbb P^u(\tau_{u,L}>t/\delta^2)\le C L^2/t,
 \qquad \delta\le L,\quad L^2\le t.
\end{equation*}

We shall also use the following killed local limit,
including its normalization.  If \(D\) is a bounded Lipschitz domain,
\(t_n=n\delta^2/3\to t\in(0,\infty)\) along admissible parities, and
\(u_\delta\to u\), \(v_\delta\to v\) in compact subsets of \(D\), then
\begin{equation}\label{eq:killed-lclt}
 \frac{Q^{\delta,n}(u_\delta,v_\delta)}{2\delta^3}
 \longrightarrow p_D(t,u,v),
\end{equation}
locally uniformly in \((t,u,v)\), and for fixed distinct \(u,v\) in one
component,
\begin{equation}\label{eq:green-normalization}
 \frac{\mathcal G_{D^\delta}(u_\delta,v_\delta)}{3\delta}
 \longrightarrow G_D(u,v).
\end{equation}
The constants two and three in these formulas are important: two is the
period of simple random walk, while admissible times have mesh
\(2\delta^2/3\).

The Gaussian bound follows from the periodic local central limit theorem
and Gaussian displacement bound
\cite[Theorem~2.1.3 and Proposition~2.1.2]{LawlerLimic}, by a half-time
semigroup split. Comparison with an enclosing cube gives a spectral gap
of order \(\delta^2\); combining this with the semigroup and Gaussian
bound gives \eqref{eq:discrete-spectral-bound}. The exit bound is the
mean exit-time estimate followed by Markov's inequality. For interior
Harnack and difference estimates we use
\cite[Theorems~6.3.8--6.3.9]{LawlerLimic}.

To obtain \eqref{eq:killed-lclt}, apply the periodic local limit theorem
and the invariance principle to the standard exit decomposition
\begin{equation}\label{eq:exit-decomposition}
 Q^{\delta,n}(u,v)=p_n^\delta(u,v)
 -\mathbb E^u\!\left[p_{n-\tau_{D^\delta}}^\delta
 (S_{\tau_{D^\delta}},v);\ \tau_{D^\delta}\le n\right].
\end{equation}
For a Lipschitz domain, exit times and locations converge by the exterior-cone estimate in \cref{sec:domain-extension}; exits within
\(\eta\) of the terminal time contribute uniformly negligibly as
\(\eta\downarrow0\), by the Gaussian bound and the positive distance
of \(v\) from the boundary. This proves the stated local uniformity.
Summing over admissible times gives \eqref{eq:green-normalization},
with the Gaussian and spectral bounds removing the duration cutoffs.

\subsection{Boundary estimates for Lipschitz domains}\label{sec:domain-extension}
For a bounded Lipschitz domain \(D\), the exterior cone condition gives constants \(r_D,c,C>0\) such that
\begin{equation}\label{eq:boundary-escape}
 \mathbb P^x\!\left(\tau_{x,r}<\tau_{D^\delta}\right)
 \le C\left(\frac{\dist(x,\partial D)+\delta}{r}\right)^c,
 \qquad x\in D^\delta,\quad \delta\le r\le r_D.
\end{equation}
Indeed, the exterior cone gives a uniformly positive probability of killing
at each geometric scale, and iteration gives \eqref{eq:boundary-escape}. We spell out the two
places where this estimate replaces the boundary estimates used in \cite{SS18}.
Let $S$ start at a lattice approximation of a fixed $a\in D$, write
$\tau=\tau_{D^\delta}$, and let
\[
  \sigma_\eta=\inf\{n\geq0:\operatorname{dist}(S_n,\partial D)\leq\eta\}.
\]
For $2\delta\leq\eta<r/8$ and $r\leq r_D$, the strong Markov property and
\eqref{eq:boundary-escape}, with adjusted constants, imply
\[
  \mathbb P\bigl(\operatorname{diam}S[\sigma_\eta,\tau]>r\bigr)
  \leq C\left(\frac{\eta+\delta}{r}\right)^c.
\]
The same
bound also controls the diameter of its loop-erasure. This supplies the boundary case of
\cite[Theorem~6.1]{SS18} and the terminal-segment estimate of
\cite[Proposition~6.6]{SS18}. The estimates near $a$ and on compact subsets of $D$ use
only local hittability and return estimates. Consequently the topological
argument of \cite[Section~6.2]{SS18} applies in the ambient compact set $J_D$:
every subsequential Hausdorff limit is a simple path from $a$ to
$\partial D$, with no other boundary contact. The possible $O(\delta)$
excursions of interpolated edges outside $D$ disappear in the limit.

For recovery, inserting \eqref{eq:boundary-escape} into the rooted-loop argument of
\cite[Proposition~5.1]{SS18} gives, for each fixed $r>0$,
\[
 \lim_{\eta\downarrow0}\limsup_{\delta\downarrow0}
 m^\delta_{D^\delta}\!\left\{
 \ell:\operatorname{diam}\ell\geq r,\
 \operatorname{dist}(\operatorname{supp}\ell,\partial D)\leq\eta
 \right\}=0.
\]
Thus the macroscopic loops may first be restricted to a compact subset of
$D$, where the interior hittability and loop-selection arguments of \cite[Section~5.2]{SS18}
apply. Removing the boundary and duration cutoffs, and then
the diameter cutoff, passes the finite recovery property to the limit;
loops of diameter less than $r$ change the decoration by at most $r$ in
Hausdorff distance. Together with convergence of the stopped walk and its
exit point, this yields
\[
  \mathcal D_1^D(K)\overset{d}=W^a[0,\tau_D].
\]
Here $\mathcal D_1^D$ denotes the direct decoration \eqref{eq:decoration} in $D$.
The constants may depend on the fixed Lipschitz domain $D$.

The recovery property in
\cite[Theorems~1.1--1.2]{SS18}, stated there for the unit ball,
extend to these domains by the same proofs: the estimates near the
starting point and in the interior are local, while
\eqref{eq:boundary-escape} supplies the boundary estimates in their
quasi-loop and terminal-segment arguments \cite[Section~6]{SS18}.
Their compactness argument then applies in the ambient space
\(J_D\); the \(O(\delta)\) excursions of interpolated edges outside
\(D\) disappear in the limit. We use these domain extensions below.

\subsection{Loop-catchers and finite recovery}
\label{sec:discrete-model}

We now recall Cai's loop-catchers and the finite recovery identities \eqref{eq:blc-discrete-recovery-transform}--\eqref{eq:blc-discrete-probe-transform} used in both main results.

Use \(D^\delta=B\cap\delta\mathbb Z^3\) and the killed transition matrix \(Q^\delta\) from \cref{sec:discrete-conventions}
and set
\begin{equation}\label{eq:blc-radial-data}
 \alpha^\delta_x=\one_{\{x=0\}},
 \qquad
 \beta^\delta_x=\frac16\#\{y\in\delta\Z^3\setminus D^\delta:|y-x|=\delta\}.
\end{equation}
The normalized excursion associated with
\((D^\delta,Q^\delta,\alpha^\delta,\beta^\delta)\) is exactly the interior
path
\begin{equation*}
 (X_0^\delta,\ldots,X_{\tau_\delta-1}^\delta),
 \qquad
 \tau_\delta=\inf\{n\ge1:X_n^\delta\notin D^\delta\},
\end{equation*}
of simple random walk started at the origin; its partition function is one.
Each step has duration \(\delta^2/3\).

For \(x\in\delta\mathbb Z^3\), write
\begin{equation*}
 \square_\delta(x)=x+[-\delta/2,\delta/2]^3,
 \qquad
 \operatorname{Cell}_\delta(S)
 =\overline B\cap\bigcup_{x\in S}\square_\delta(x).
\end{equation*}
Let \(S_\lambda^\delta\) be the range of the time-parametrized
\(\lambda\)-random-walk loop-catcher of \cite[Algorithm~2.5 and
Definition~2.6]{Cai26} for the setup \eqref{eq:blc-radial-data}.  Conditional
on \(X_{\tau_\delta-1}^\delta=x\), choose an exterior neighbour
\(Y_\delta\) with probability \((1/6)/\beta_x^\delta\), and let
\(\zeta_\delta\) be the first point of \([x,Y_\delta]\) on \(\partial B\).
Using the same connector in the catcher and the original-walk trace, put
\begin{equation}\label{eq:blc-geometric-traces}
 K_\lambda^\delta
 =\operatorname{Cell}_\delta(S_\lambda^\delta)
       \cup[X_{\tau_\delta-1}^\delta,\zeta_\delta],
 \qquad
 R^\delta
 =\operatorname{Cell}_\delta\{X_0^\delta,\ldots,
                         X_{\tau_\delta-1}^\delta\}
       \cup[X_{\tau_\delta-1}^\delta,\zeta_\delta].
\end{equation}
Changing the connector by a set of diameter \(O(\delta)\) does not change
any scaling limit. For nonconvex \(B\), the clipped cell trace remains at Hausdorff distance \(O(\delta)\) from the connected ambient cell trace, so connectedness is preserved in the limit.  Loop masses below act on vertex supports; Hausdorff
convergence acts on the cell traces \eqref{eq:blc-geometric-traces}.

We use the multipath version as well.  For a finite ordered list of
interior endpoint pairs \((a_i,b_i)_{i\le r}\), with \(a_i\ne b_i\) and
repetitions of pairs allowed, and for lattice approximations,
let \(S_{\lambda,\times}^\delta\) denote the range returned by
\cite[Algorithm~2.5]{Cai26} from independent normalized random-walk bridges
\(a_i^\delta\to b_i^\delta\).  Write
\begin{equation*}
 C_{\lambda,\times}^\delta
 =\operatorname{Cell}_\delta(S_{\lambda,\times}^\delta).
\end{equation*}

The finite-graph results we need are the following.  First,
\(S_\lambda^\delta\) is connected, and
\begin{equation}\label{eq:blc-discrete-recovery-transform}
 \mathbb E\!\left[
  \one_{\{S_\lambda^\delta\cap F=\varnothing\}}
  e^{-\lambda m_{D^\delta}^{\delta}(S_\lambda^\delta,F)}\right]
 =\mathbb P\!\left[
  \{X_0^\delta,\ldots,X_{\tau_\delta-1}^\delta\}\cap F=\varnothing
  \right]
\end{equation}
for every \(F\subset D^\delta\).  Second,
\begin{equation}\label{eq:blc-discrete-probe-transform}
 \mathbb E\!\left[
  \one_{\{S_{\lambda,\times}^\delta\cap F=\varnothing\}}
  e^{-\lambda m_{D^\delta}^{\delta}(S_{\lambda,\times}^\delta,F)}\right]
 =\prod_{i=1}^r
   \frac{\mathcal G_{D^\delta\setminus F}(a_i^\delta,b_i^\delta)}
        {\mathcal G_{D^\delta}(a_i^\delta,b_i^\delta)}.
\end{equation}
The first identity is \cite[Theorem~1.3 and (1.7)]{Cai26}, and the
second is \cite[(2.10)]{Cai26}.  Algorithm~2.7 of \cite{Cai26}
recovers paths from the time-parametrized catcher and an independent
intensity-\(\lambda\) random-walk loop soup.  Proposition~2.8 of \cite{Cai26}
identifies their joint law: in the radial setup the recovered path
has the law of the original interior path
\((X_0^\delta,\ldots,X_{\tau_\delta-1}^\delta)\),
and in the multipath setup the recovered paths are independent
normalized random-walk bridges.

\subsection{Compact sets and measurable kernels}
In this subsection, we specify the measurable spaces and kernels needed to recover a random compact set from its Green coordinates.

Let \(\K(\overline B)\) be the space of nonempty compact subsets of
\(\overline B\), equipped with Hausdorff distance; it is compact Polish. 
We also allow the empty set as an isolated point when taking
intersections. The class \(\mathfrak C_B\) defined in the introduction
is Borel: connectedness and containment of \(0\) are closed conditions,
and the condition that \(K\cap\partial B\) is a singleton is Borel.

We set \(G_{B\setminus K}(x,y)=0\) if either endpoint belongs to
\(K\) or if the endpoints lie in different components of
\(B\setminus K\). Note that for fixed distinct \(x,y\in B\) and open \(O\),
the maps
\begin{equation}\label{eq:borel-basic}
 K\longmapsto\dist(K,O^c),\qquad
 K\longmapsto\NCap(K\cap\overline O),\qquad
 K\longmapsto G_{B\setminus K}(x,y)
\end{equation}
are Borel,
and for every fixed finite \(\lambda\ge0\), the symmetric kernel
\begin{equation}\label{eq:joint-kernel-borel}
 H_\lambda(K,A):=
 \one_{\{K\cap A=\varnothing\}}e^{-\lambda m_B(K,A)}
\end{equation}
is jointly Borel on \(\K(\overline B)^2\).

\subsection{Recovery property and local nonpolarity}\label{sec:nonpolar}

We collect the potential-theoretic arguments needed for both LERW and
loop-catchers. A compact set is polar precisely when its Newtonian
capacity is zero \cite[Theorem~8.20]{MortersPeres}.
Let \(\mathscr T\) be the countable family of triples of rational
balls satisfying
\[
 \overline W\subset V,\qquad \overline V\subset U,
 \qquad \overline U\subset B.
\]
Call \(K\) \emph{locally nonpolar} if
\begin{equation}\label{eq:local-nonpolarity}
 K\cap W\ne\varnothing
 \quad\Longrightarrow\quad
 \NCap(K\cap\overline U)>0
 \qquad ((W,V,U)\in\mathscr T).
\end{equation}
Let \(\mathfrak C_B^{\mathrm{np}}\) consist of the sets in
\(\mathfrak C_B\) with empty interior that are locally nonpolar.
This is a standard Borel space: empty interior is tested by rational
balls, and \eqref{eq:local-nonpolarity} is a countable Borel condition
by \eqref{eq:borel-basic}.
We also write $\Gamma^+=\Gamma\cap\mathfrak C_B^{\mathrm{np}}$.

The next lemma states that the recovery property \eqref{eq:lerw-recovery}, and its general-intensity version \eqref{eq:looprecovery}, force local nonpolarity.
\begin{lemma}
\label{lem:no-polar-subarc}\label{lem:blc-nonpolarity}
Fix \(0\le\lambda<\infty\). For a random \(K\in\mathfrak C_B\),
the identities
\begin{equation}\label{eq:shared-recovery-transform}
 \mathbb E H_\lambda(K,A)
 =\mathbb P^0(W[0,\tau_B]\cap A=\varnothing),
 \qquad A\Subset B\text{ compact},
\end{equation}
are equivalent to
\(\mathcal D_\lambda(K)\stackrel d=W[0,\tau_B]\).
Either condition implies
\(K\in\mathfrak C_B^{\mathrm{np}}\) almost surely; if
\(K\in\Gamma\), it implies \(K\in\Gamma^+\) almost surely.
\end{lemma}

We omit the verification of the equivalence, which follows from Fact \ref{lem:closed-decoration-avoidance} and
\cite[Section~1.1]{Molchanov}, since both compact sets are the closures
of their interior traces.

\begin{proof}[Proof of local nonpolarity]
We give the local argument outlined in \cite[Section~1.5]{Cai26}.
First, \(K\subset\mathcal D_\lambda(K)\) and a three-dimensional
Brownian range has empty interior, so \(K\) has empty interior.
If local nonpolarity fails with positive probability, countability of
\(\mathscr T\) gives a fixed triple \((W,V,U)\) for which
\[
 K\cap W\ne\varnothing,\qquad \NCap(K\cap\overline U)=0
\]
with positive probability. Conditional on such a \(K\), Brownian loops
almost surely avoid the polar set \(K\cap\overline U\), by local
absolute continuity of bridges with respect to Brownian motion.
Consequently every attached loop entering \(V\) also hits
\(K\setminus U\), and with conditional probability at least
\[
 \exp\{-\lambda m_B(\overline V,\overline B\setminus U)\}>0
\]
no such loop occurs, by \eqref{eq:separated-loop-finite}.
On this event the decoration meets \(W\), but its intersection with
\(\overline W\) is polar.
This is impossible for a stopped Brownian range: whenever it meets the
open set \(W\), it contains there the range of a nondegenerate time
interval, which has dimension two and positive Newtonian capacity
\cite[Theorems~4.29(b) and~4.32]{MortersPeres}.
\end{proof}

\subsection{Green coordinates and identification of laws}
\label{sec:injectivity}\label{sec:moments}

Fix \(\mathcal Z=\Q^3\cap B\), and set
\(\mathcal Z^{(2)}=\{(x,y)\in\mathcal Z^2:x\ne y\}\).
For \(K\in\K(\overline B)\), define
\begin{equation}\label{eq:green-coordinate}
 X_{x,y}(K)=\frac{G_{B\setminus K}(x,y)}{G_B(x,y)}\in[0,1],
 \qquad (x,y)\in\mathcal Z^{(2)},
\end{equation}
and the Borel map
\begin{equation}\label{eq:Phi-map}
 \Phi(K)=\bigl(X_{x,y}(K)\bigr)_{(x,y)\in\mathcal Z^{(2)}}.
\end{equation}

The next proposition says that for locally nonpolar sets, the Green coordinates determine the set. We omit the proof, which follows \cite[Lemma~4.4]{Cai26}, with local
nonpolarity providing positive Newtonian capacity in place of the
positive logarithmic capacity of a nondegenerate planar continuum, as
outlined in \cite[Section~1.5]{Cai26}.

\begin{proposition}
\label{prop:green-injective}\label{prop:blc-green-injective}
The map \(\Phi\) is injective on \(\mathfrak C_B^{\mathrm{np}}\),
and in particular on \(\Gamma^+\).
\end{proposition}

\begin{remark}[Measurable reconstruction]\label{rem:borel-inverse}
The Lusin--Suslin theorem states that an injective Borel map between
standard Borel spaces has Borel image and a Borel inverse on its
image; see \cite[Theorem~15.1 and Corollary~15.2]{Kechris}. Applied to \(\Phi\) on
\(\mathfrak C_B^{\mathrm{np}}\), it therefore implies that the law
of the coordinate vector \(\Phi(K)\) determines the law of \(K\)
as a compact set. This is the only measurable reconstruction result
needed below.
\end{remark}

The following lemma identifies the law of the random compact set once probes realizing products of Green coordinates are available.
\begin{lemma}
\label{lem:mixed-moments}\label{lem:blc-mixed-moments}
Fix \(0\le\lambda<\infty\). Suppose that, for each finite list
\((x_i,y_i)_{i\le r}\subset\mathcal Z^{(2)}\), with repetitions
allowed, there is a random compact \(C\Subset B\) such that
\begin{equation}\label{eq:shared-probe-hypothesis}
 \mathbb E_C H_\lambda(C,A)
 =\prod_{i=1}^rX_{x_i,y_i}(A),
 \qquad A\subset\overline B\text{ compact}.
\end{equation}
Then at most one law on \(\mathfrak C_B\) satisfies
\eqref{eq:shared-recovery-transform}. In particular, any two such
laws \(\mu_1,\mu_2\) have the same mixed Green function moments:
\begin{equation}\label{eq:mixed-moment-equality}
 \mathbb E_{\mu_1}\prod_{i=1}^rX_{x_i,y_i}(K)
 =\mathbb E_{\mu_2}\prod_{i=1}^rX_{x_i,y_i}(K).
\end{equation}
\end{lemma}

We omit the proof because, under the stated probe hypothesis, the
Fubini and moment-determinacy argument of \cite[Lemma~4.3 and the proof
of Theorem~1.5]{Cai26} applies once
\cref{lem:no-polar-subarc,prop:green-injective,rem:borel-inverse} are
available.

\subsection{The finite interior probe}\label{sec:finite}
In this subsection, we explain how entangled multipath loop erasure provides the probes that turn the recovery transform \eqref{eq:recovery-transform} into products of Green functions.

We use the entangled multipath LERW and its recovery property \eqref{eq:finite-transform} from
\cite[Definition~2.1, Proposition~2.2 and (2.6)]{Cai26}.
We recall the construction and the identities needed below; their
finite-graph proofs are contained in that reference.

  Let
\(V\) be a finite set and let \(Q=(Q_{xy})_{x,y\in V}\) be symmetric,
nonnegative, and substochastic with \(\rho(Q)<1\).  For \(R\subset V\), put
\begin{equation*}
 \mathcal G_R=(I-Q_R)^{-1},\qquad
 d_R=\det(I-Q_R),\qquad d_\varnothing=1.
\end{equation*}
We set \(\mathcal G_R(a,b)=0\) whenever \(a\notin R\) or \(b\notin R\).
For a walk or directed loop \(\xi\), let \(q(\xi)\) be its edge-weight
product and \(\supp(\xi)\) its vertex set, with empty product one.

For a directed walk \(\omega\), chronological loop erasure produces its
\emph{LERW record}
\begin{equation*}
 P(\omega)=(\eta;\ell_1,\ldots,\ell_m),\qquad \eta=\LE(\omega),
\end{equation*}
where the \(\ell_j\)'s are the directed simple loops in erasure order;
loops of length one are included.  Directed loops are identified under
cyclic shifts, but not under reversal, and repeated copies are distinguished.
Records are taken modulo adjacent swaps of erased loops having disjoint
supports.  On the erased loops and the terminal path \(\eta\), let
\(\prec\) be the transitive closure of
\begin{equation*}
 \ell_i\prec\ell_j
 \quad\text{if }i<j\text{ and }
 \supp(\ell_i)\cap\supp(\ell_j)\ne\varnothing,
 \qquad
 \ell_i\prec\eta
 \quad\text{if }\supp(\ell_i)\cap\supp(\eta)\ne\varnothing.
\end{equation*}
Following \cite[Section~2.1]{Cai26}, \(\prec\) is the \emph{entangled
order}; it is unchanged by the allowed swaps.  A record is legal if every
erased loop \(\ell\) satisfies \(\ell\prec\eta\).  Chronological loop
erasure is a weight-preserving bijection between directed walks and legal
records.  Put
\begin{equation*}
 q(P)=q(\eta)\prod_{j=1}^m q(\ell_j),\qquad
 \supp(P)=\supp(\eta)\cup\bigcup_{j=1}^m\supp(\ell_j).
\end{equation*}

For \(A\subset V\), extend the entangled order by declaring
\(A\prec\ell_j\) when \(A\cap\supp(\ell_j)\ne\varnothing\), and again take
the transitive closure.  Inspect the loops in erasure order, starting from
\(A_0=A\): keep \(\ell_j\) and set
\(A_j=A_{j-1}\cup\supp(\ell_j)\) when
\(\supp(\ell_j)\cap A_{j-1}\ne\varnothing\), and otherwise discard it and
put \(A_j=A_{j-1}\).  The record keeping \(\eta\) and precisely the loops
\(\ell\) for which \(A\prec\ell\) is denoted by \(\Phi_A(P)\).  This is the
\emph{reduction of the LERW record} relative to \(A\), and is independent of
the representative of the record.  Let \(\mathcal R_A^R(a,b)\) be the set
of legal records from \(a\) to \(b\) in \(R\) satisfying
\(\Phi_A(P)=P\).

Fix ordered pairs \((a_i,b_i)\), \(1\le i\le r\), with \(a_i\ne b_i\).
Repeated pairs are allowed.  Starting with \(C_0=\varnothing\), sample
independent normalized random-walk bridges \(\omega_i:a_i\to b_i\), assuming
\(\mathcal G_V(a_i,b_i)>0\) for every \(i\), with law
\begin{equation*}
 \mathbb P(\omega_i=\omega)
 =\frac{q(\omega)}{\mathcal G_V(a_i,b_i)},
 \qquad \omega:a_i\to b_i\text{ in }V.
\end{equation*}
Set
\begin{equation}\label{eq:entangled-construction}
 H_i=P(\omega_i),\qquad
 P_i=\Phi_{C_{i-1}}(H_i),\qquad
 C_i=C_{i-1}\cup\supp(P_i).
\end{equation}

Let \(\gamma_i\) be the directed path corresponding to the legal
record \(P_i\).  In the terminology of \cite[Definition~2.1]{Cai26},
\((\gamma_1,\ldots,\gamma_r)\) is the \emph{entangled multipath LERW}
and \(C_r=\bigcup_i\supp(\gamma_i)\) is its range.  For \(r=1\),
\(\gamma_1=\LE(\omega_1)\).

Let \(\mu_V^{\mathrm{loop}}\) be the push-forward under cyclic re-rooting of
the measure assigning a rooted loop \((z_0,\ldots,z_n=z_0)\) weight
\(n^{-1}q(\ell)\), and put
\begin{equation*}
 m_V(A,E)=\mu_V^{\mathrm{loop}}
 \{\ell:\supp\ell\cap A\ne\varnothing,\ \supp\ell\cap E\ne\varnothing\}.
\end{equation*}
For an independent intensity-one random-walk loop soup \(\mathcal L_V^1\),
define
\begin{equation*}
 \mathcal D_1^V(A)
 =A\cup\bigcup_{\ell\in\mathcal L_V^1:\,\supp\ell\cap A\ne\varnothing}
          \supp\ell.
\end{equation*}

The next proposition shows that the finite recovery property \eqref{eq:finite-transform} expresses the avoidance transform of the entangled range as a product of Green-function ratios. We omit the proof, since the identities are precisely \cite[Proposition~2.2 and (2.6)]{Cai26}, with the
ambient set denoted here by \(V\) and the allowed subset by \(R\).
\begin{proposition}\label{prop:finite-probe}
For every \(R\subset V\),
\begin{equation}\label{eq:finite-transform}
 \mathbb E\!\left[
   \one_{\{C_r\subset R\}}e^{-m_V(C_r,V\setminus R)}
 \right]
 =\prod_{i=1}^r\frac{\mathcal G_R(a_i,b_i)}{\mathcal G_V(a_i,b_i)}.
\end{equation}
Equivalently,
\begin{equation}\label{eq:finite-decoration-law}
 \mathcal D_1^V(C_r)\ \stackrel d=\
 \bigcup_{i=1}^r\supp(\omega_i).
\end{equation}
\end{proposition}

\subsection{Convergence of normalized bridges}\label{sec:bridge-convergence}

Let \(D\subset\R^3\) be a bounded Lipschitz domain. For a finite list of
ordered pairs \((x_i,y_i)\) of distinct points in the same component
of \(D\), take independent normalized random-walk bridges
\(\omega_i^\delta\) between lattice approximations and put
\(C^{0,\delta}=\bigcup_i\tr(\omega_i^\delta)\).
The following is the three-dimensional counterpart of
\cite[Lemma~3.1 and Corollary~3.2]{Cai26}, with convergence strengthened
to the variable-duration uniform topology; it is used for both LERW and
loop-catcher probes.

\begin{lemma}
\label{lem:green-path-limit}
The interpolated paths \((\omega_i^\delta)_{i\le r}\) converge jointly to
independent normalized Brownian bridges
\((\Gamma_D^{x_i,y_i})_{i\le r}\).  In
particular,
\begin{equation}\label{eq:raw-union-limit}
 C^{0,\delta}\Rightarrow
 C^0:=\bigcup_{i=1}^r\tr(\Gamma_D^{x_i,y_i})
\end{equation}
in Hausdorff distance, and \(C^0\Subset D\) almost surely.
\end{lemma}

\begin{proof}
We give only the additional points needed for path convergence and the
three-dimensional normalization. For one pair of endpoints, let \(N_\omega\) be the number of steps of the bridge. Then
\begin{equation}\label{eq:discrete-duration-law}
 \mathbb P_{D^\delta}^{x_\delta,y_\delta}(N_\omega=n)
 =\frac{Q^{\delta,n}(x_\delta,y_\delta)}
        {\mathcal G_{D^\delta}(x_\delta,y_\delta)}.
\end{equation}
On a fixed window \(n\delta^2/3\in[t_0,T]\), the killed local limit
gives convergence of bridge finite-dimensional distributions.
The bridge density relative to walk through time \(k\le n/2\) is
\[
 \frac{Q^{\delta,n-k}(S_k,y_\delta)}
      {Q^{\delta,n}(x_\delta,y_\delta)}\le C_{t_0,T},
\]
by \eqref{eq:lclt} and \eqref{eq:killed-lclt}. Tightness from Donsker's invariance principle and reversibility gives tightness in the uniform topology 
on both halves. Thus, on each fixed duration window, 
the laws of the fixed-duration bridges converge 
weakly in the uniform topology.

By \eqref{eq:killed-lclt} and \eqref{eq:green-normalization}, the duration law
converges to \(p_D(t,x,y)\,\dd t/G_D(x,y)\), since admissible times
have mesh \(2\delta^2/3\). The Gaussian and spectral bounds make the
short- and long-duration tails uniformly negligible, so removing the
window gives the claimed variable-duration path convergence.
Independence gives the joint statement; taking traces gives
\eqref{eq:raw-union-limit}, and each limiting bridge has a compact trace
contained in \(D\).
\end{proof}

\section{Hittability of the entangled multipath LERW}\label{sec:hittability-input}

This section extends the ordinary-LERW hittability estimates of
Sapozhnikov--Shiraishi~\cite[Lemmas~3.2--3.3]{SS18} to entangled multipath
LERW, which provides the main technical input in the outline of~\cite[Section 1.5]{Cai26}. Cai's initial-segment inclusion reduces the geometric part to
ordinary loop erasure; the additional work is to transfer the estimates
to normalized bridges, uniformly near their marked endpoints.

We now work in a fixed bounded Lipschitz domain \(D\subset\R^3\); later
\(D=B\).  Fix interior ordered pairs \((a_i,b_i)_{i\le r}\) with
\(a_i\ne b_i\), allowing repeated pairs, choose lattice approximations,
and construct \(C_r^\delta\) by \eqref{eq:entangled-construction}.  Write
\(N_i=\operatorname{len}\omega_i\) for the number of steps of the normalized
random-walk bridge \(\omega_i\).  We use the balls and exit times from
\eqref{eq:lattice-ball-exit}.  Let \(S\) be a simple random walk, independent
of all normalized random-walk bridges, and define
\begin{equation}\label{eq:hittability-functional}
 \cH_\delta(C;s,L;U)
 :=\sup_{\substack{z\in U\cap\delta\Z^3\\ \dist(z,C)\le s}}
 \mathbb P^z\big(S[0,\tau_{z,L}]\cap C=\varnothing\big),
 \qquad U\Subset D.
\end{equation}

The next proposition says that a walk started close to the entangled range is likely to hit it before leaving a slightly larger neighbourhood.
\begin{proposition}
\label{prop:multipath-hittability}
For every \(U\Subset D\) and
\(0<R_*<\dist(U,\partial D)/10\), there is \(\eta_H>0\) such that
\begin{equation}\label{eq:H2}
 \lim_{\eps\downarrow0}\limsup_{\delta\downarrow0}
 \mathbb P\left(
  \cH_\delta(C_r^\delta;\eps^2R_*,\sqrt\eps R_*;U)
       >\eps^{\eta_H}\right)=0.
\end{equation}

\end{proposition}

The proof below adapts \cite[Lemmas~3.2--3.3]{SS18} to the entangled multipath LERW from normalized random-walk bridges using \cite[Lemma~2.4]{Cai26}.
The additional work is the
transfer from ordinary walks to normalized bridges, including uniform
control near their marked endpoints. We keep these steps and their
scale choices, which are needed for the near-miss estimate.

The next lemma, \cite[Lemma~2.4]{Cai26}, states that every vertex of the entangled range is the endpoint of a loop-erased initial segment contained in that range.

\begin{lemma}\label{lem:prefix-witness}
For every \(w\in C_r^\delta\), there are an index \(i\le r\) and a time index
\(\sigma\) of the normalized random-walk bridge
\(\omega_i=(X_0^i,\ldots,X_{N_i}^i)\) such that
\begin{equation}\label{eq:prefix-witness}
 X_\sigma^i=w,
 \qquad \supp\LE(X^i[0,\sigma])\subset C_r^\delta.
\end{equation}
\end{lemma}

We also need a familiar last-exit version for the final loop erasure. We omit the proof, which follows immediately from the
last-visit construction of chronological loop erasure.

\begin{lemma}\label{lem:last-exit-prefix}
Let \(\eta=\LE(X[0,N])\), and let \(u\) be a vertex of \(\eta\).  There
is \(\sigma\le N\) with \(X_\sigma=u\) such that
\(\LE(X[0,\sigma])\) is the initial segment of \(\eta\) ending at \(u\).
\end{lemma}

We are now ready to prove  \cref{prop:multipath-hittability}.

\begin{proof}[Proof of \cref{prop:multipath-hittability}]
\emph{Step 1: scales, lattice cover, and quantifiers.}
Put \(d_i^0=|a_i-b_i|>0\).  Fix an integer \(M\ge1\), to be chosen at
the end, and put
\begin{equation}\label{eq:three-scales}
 s=\eps^2R_*,\qquad h=\eps^{5/4}R_*,\qquad
 L=\eps^{1/2}R_*.
\end{equation}
Thus
\begin{equation}\label{eq:scale-ratios}
 \frac{s}{h}=\frac{h}{L}=\eps^{3/4}.
\end{equation}
The order of all choices below is important.  We first choose \(\eps>0\)
so small that
\begin{equation}\label{eq:scale-separation}
 32s<h,\qquad 32h<L,
\end{equation}
and that \(L\) is smaller than one twentieth of
\[
 \min_{i\le r}\big\{\dist(a_i,\partial D),\dist(b_i,\partial D),
                         |a_i-b_i|\big\}.
\]
We then choose \(\delta_0(\eps)>0\) so small that, for every
\(0<\delta<\delta_0(\eps)\), one has \(\delta<s/100\), and the same
inequalities, with a factor \(1/2\) of slack, hold for all lattice
approximations of the endpoints.  Shrinking \(\delta_0(\eps)\) once more,
we also require, for every \(i\),
\begin{equation}\label{eq:endpoint-uniform-separation}
 |a_i^\delta-b_i^\delta|\ge\frac12d_i^0,
 \qquad R+\delta\le\frac14d_i^0,
 \quad R:=L/2.
\end{equation}
From now on \(\delta\) is arbitrary in this range.  This is exactly the
order of limits in \eqref{eq:H2}.

Since \(\dist(U,\partial D)>10R_*\), the closed \(2R_*\)-neighbourhood
\(U^+\) of \(U\) is compactly contained in \(D\).  Put
\[
 s_\delta=2\delta\left\lceil\frac{s}{2\delta}\right\rceil.
\]
The half-open cubes
\[
 Q=v_Q+[-s_\delta/2,s_\delta/2)^3,\qquad
 v_Q\in s_\delta\Z^3\subset\delta\Z^3,
\]
partition \(\R^3\).  Let \(\mathcal Q_\delta\) consist of the cubes
meeting the closed \((s+\delta)\)-neighbourhood of \(U\).  Then
\begin{equation}\label{eq:cube-cover-count}
 |\mathcal Q_\delta|\le C_U s^{-3},\qquad
 Q\subset B(v_Q,s)\quad(Q\in\mathcal Q_\delta),
\end{equation}
for all \(\delta<s/100\).  Indeed,
\(s\le s_\delta<s+2\delta\), and the half diagonal of \(Q\) is
\(\sqrt3s_\delta/2<s\).  The second inclusion also shows that all
\(B(v_Q,h/2)\) occurring below are contained in \(U^+\), hence in \(D\).

For \(Q\in\mathcal Q_\delta\) and a nearest-neighbour path \(\gamma\),
set
\begin{align}
 \sigma_Q(\gamma)&=\inf\{j:\gamma_j\in \B_\delta(v_Q,4s)\},
       \notag\\
 \Lambda_Q(\gamma)&=\supp(\gamma[0,\sigma_Q])\cap
       \big(\B_\delta(v_Q,h/2)\setminus\B_\delta(v_Q,4s)\big),
       \label{eq:Gamma-Q}\\
 \Psi_Q(\gamma)&=\sup_{z\in \B_\delta(v_Q,3s)}
 \mathbb P^z\big(S[0,\tau_{v_Q,h/2}]\cap\Lambda_Q(\gamma)=\varnothing\big).
 \notag
\end{align}
When \(\sigma_Q(\gamma)=\infty\), we set \(\Psi_Q(\gamma)=0\).

\emph{Step 2: the off-center initial-segment estimate.}
The following is \cite[Lemma~3.3]{SS18}, in the notation needed below.

\begin{claim}\label{clm:off-center-prefix}
There are \(A_M,C_M<\infty\) and \(\kappa_M>0\), depending only on
\(M\), such that, whenever \(|a-v_Q|>2h\),
\begin{equation}\label{eq:SS-off-center}
 \mathbb P^a\big(\exists T<\tau_{D^\delta}:S_T\in Q,\
 \Psi_Q(\LE(S[0,T]))>A_M(s/h)^{\kappa_M}\big)
 \le C_M(s/h)^M.
\end{equation}
\end{claim}

\begin{proof}
Apply \cite[Lemma~3.3]{SS18} after translation and scaling, with inner
and outer radii \(4s/\delta\) and \(h/(2\delta)\).  The portion of the loop-erased path used in the
annular avoidance event of that lemma
is contained in \(\Lambda_Q\), and interior Harnack replaces the
starting point \(v_Q\) by the supremum over \(\B_\delta(v_Q,3s)\).
The radii have ratio \(8s/h\); absorbing the fixed factors into
\(A_M,C_M\) gives \eqref{eq:SS-off-center}.
\end{proof}

We also record the following identity for stopped normalized-random-walk bridges.

\begin{claim}\label{clm:stopped-Green}
Let \(\tau\) be a stopping time for the killed walk in \(D^\delta\), and
let \(F\in\mathcal F_\tau\).  If \(N\) is the number of steps of a
normalized random-walk bridge from \(a\) to \(b\), then
\begin{equation}\label{eq:stopped-Doob-exact}
 \mathbb P_{D^\delta}^{a,b}(F,\tau<N)
 =\frac{\mathbb E^a\!\left[
   \one_{\{F,\tau<\tau_{D^\delta}\}}
   \big(\mathcal G_{D^\delta}(S_\tau,b)-\one_{\{S_\tau=b\}}\big)\right]}
 {\mathcal G_{D^\delta}(a,b)}.
\end{equation}
\end{claim}

\begin{proof}
Suppose first that \(\tau\) is bounded.  Fix a stopped initial segment
\(\xi=(a=\xi_0,\ldots,\xi_k=x)\) on which \(F\) occurs.  Every full path
counted by \(\{\tau<N\}\) is obtained uniquely by appending to \(\xi\) a
positive-length path from \(x\) to \(b\).  The total weight of all such
continuations is
\[
 \sum_{n\ge1}Q^{\delta,n}(x,b)
 =\mathcal G_{D^\delta}(x,b)-\one_{\{x=b\}}.
\]
Multiply by the weight of \(\xi\), sum over stopped initial segments, and divide
by \(\mathcal G_{D^\delta}(a,b)\).  This proves
\eqref{eq:stopped-Doob-exact} for bounded \(\tau\).  Applying it to
\(\tau\wedge m\) on \(\{\tau\le m\}\), and then using monotone
convergence, proves the general case.
\end{proof}

For any path \(\xi\), define
\begin{equation}\label{eq:tau-Q-definition}
 \tau_Q(\xi)=\inf\big\{T\ge0:\xi_T\in Q,\
 \Psi_Q(\LE(\xi[0,T]))>A_M(s/h)^{\kappa_M}\big\}.
\end{equation}
It is a stopping time because the expression tested at time \(T\) is a
deterministic function of \(\xi[0,T]\).  Call \(Q\) \emph{\(i\)-far} if
\begin{equation}\label{eq:i-far}
 |v_Q-a_i^\delta|>2h\quad\hbox{and}\quad
 |v_Q-b_i^\delta|>2h.
\end{equation}
On \(\{\tau_Q<\tau_{D^\delta}\}\), the stopped point is in
\(Q\subset B(v_Q,s)\).  For an \(i\)-far cube,
\[
 |S_{\tau_Q}-b_i^\delta|>2h-s>h,
\]
so the indicator correction in \eqref{eq:stopped-Doob-exact} vanishes.
The stopped identity and \cref{clm:off-center-prefix} now give
\begin{equation}\label{eq:stopped-Doob}
 \mathbb P_{D^\delta}^{a_i^\delta,b_i^\delta}(\tau_Q<N_i)
 \le
 \frac{\sup_{x\in Q\cap\delta\Z^3}\mathcal G_{D^\delta}(x,b_i^\delta)}
      {\mathcal G_{D^\delta}(a_i^\delta,b_i^\delta)}
 \mathbb P^{a_i^\delta}(\tau_Q<\tau_{D^\delta}).
\end{equation}
By \eqref{eq:green-3d-bounds}, the numerator is at most
\(C\delta/h\), whereas the denominator is at least \(c\delta\).
Consequently
\begin{equation}\label{eq:far-cube-bound}
 \mathbb P_{D^\delta}^{a_i^\delta,b_i^\delta}(\tau_Q<N_i)
 \le C_Mh^{-1}(s/h)^M.
\end{equation}

\emph{Step 3: cubes near the two endpoints.}
For \(x\in D^\delta\) and
\(0<\rho<R\), define

\begin{equation*}
 \Psi^{\mathrm{cen}}_{x;\rho,R}(\gamma)=
 \sup_{z\in\B_\delta(x,\rho/2)}
 \mathbb P^z\big(S[0,\tau_{x,R}]\cap\supp(\gamma)\cap(\B_\delta(x,R)\setminus\B_\delta(x,\rho))=\varnothing\big).
\end{equation*}

The next claim gives the centered analogue of \cref{clm:off-center-prefix}, controlling avoidance of the annular part of a loop-erased initial segment that exits the outer ball.

\begin{claim}\label{clm:centered-prefix}
There are \(A_M^{\mathrm e},C_M<\infty\) and
\(\kappa_{\mathrm e,M}>0\) such that, whenever
\(0<\rho<R\), \(\delta\le\rho/10\), and \(B(x,R)\subset D\),
\[
 \mathbb P^x\big(\exists T<\tau_{D^\delta}:\supp\LE(S[0,T])
 \not\subset\B_\delta(x,R),
 \Psi^{\mathrm{cen}}_{x;\rho,R}(\LE(S[0,T]))
 >A_M^{\mathrm e}(\rho/R)^{\kappa_{\mathrm e,M}}\big)
 \le C_M(\rho/R)^M.
\]
\end{claim}

\begin{proof}
This is \cite[Lemma~3.2]{SS18} after translation and scaling, followed
by interior Harnack to allow all starting points in
\(\B_\delta(x,\rho/2)\).  Its annular target is contained in ours,
so the same bound applies after adjusting the constants.
\end{proof}

From now on set
\begin{equation*}
 \rho=8h,\qquad R=L/2,\qquad \frac{\rho}{R}=\frac{16h}{L},
\end{equation*}
and put
\begin{equation*}
 B_M=A_M^{\mathrm e}16^{\kappa_{\mathrm e,M}}.
\end{equation*}

Let \(\eta_i=\LE(\omega_i)\).  For
\(e\in\{a_i^\delta,b_i^\delta\}\), let \(\mathcal B_{i,e}\) be the event
\begin{equation}\label{eq:endpoint-bad}
 \sup_{z\in D^\delta:\,|z-e|\le3h}\mathbb P^z
       (S[0,\tau_{z,L}]\cap\supp(\eta_i)=\varnothing)
 >B_M(h/L)^{\kappa_{\mathrm e,M}}.
\end{equation}

The next claim transfers this centered estimate to normalized random-walk bridges, uniformly near either marked endpoint.
\begin{claim}\label{clm:endpoint-bound}
Uniformly in sufficiently small \(\eps\) and then sufficiently small
\(\delta\),
\begin{equation}\label{eq:endpoint-bound}
 \mathbb P(\mathcal B_{i,a_i^\delta}\cup\mathcal B_{i,b_i^\delta})
 \le C_M(h/L)^M.
\end{equation}
\end{claim}

\begin{proof}
We first treat \(a=a_i^\delta\), \(b=b_i^\delta\), and suppress \(i\).
The scale choice in Step~1 ensures
\begin{equation}\label{eq:endpoint-geometric-room}
 B(a,L)\Subset D,\qquad |a-b|>4L.
\end{equation}
Since \(\eta_i\) joins \(a\) to \(b\), it leaves the closed lattice ball
\(\B_\delta(a,R)\), where \(R=L/2\).  Let
\begin{equation*}
 k=\inf\{j:\eta_i(j)\notin\B_\delta(a,R)\},\qquad u=\eta_i(k).
\end{equation*}
The
nearest-neighbour property gives
\begin{equation}\label{eq:first-exit-shell}
 R<|u-a|\le R+\delta.
\end{equation}
By \cref{lem:last-exit-prefix}, applied to the final loop erasure, there
is a time index \(\sigma\) such that
\begin{equation}\label{eq:endpoint-last-exit-witness}
 \omega_{i,\sigma}=u,\qquad
 \LE(\omega_i[0,\sigma])=\eta_i[0,k].
\end{equation}
Because \(u\ne b\) by \eqref{eq:endpoint-geometric-room},
\(\sigma<N_i\).

Suppose \(\mathcal B_{i,a}\) occurs.  The supremum in
\eqref{eq:endpoint-bad} is over a finite set, so there is
\(z\in D^\delta\), \(|z-a|\le3h\), for which
\begin{equation}\label{eq:endpoint-witness-z}
 \mathbb P^z\big(S[0,\tau_{z,L}]\cap\supp(\eta_i)=\varnothing\big)
 >B_M(h/L)^{\kappa_{\mathrm e,M}}.
\end{equation}
The scale separation gives
\begin{equation}\label{eq:endpoint-ball-inclusions}
 z\in\B_\delta(a,4h)=\B_\delta(a,\rho/2),\qquad
 \B_\delta(a,R)\subset\B_\delta(z,L),
\end{equation}
because \(R+3h<L\).  Let
\[
 \mathcal A_z=\{S[0,\tau_{z,L}]\cap\supp(\eta_i)=\varnothing\}.
\]
On \(\mathcal A_z\), the walk avoids every subset of \(\eta_i\) up to the
earlier exit time \(\tau_{a,R}\).  In particular,
\begin{equation}\label{eq:endpoint-event-inclusion-detailed}
 \mathcal A_z\subset\big\{
 S[0,\tau_{a,R}]\cap\supp(\eta_i[0,k])
 \cap(\B_\delta(a,R)\setminus\B_\delta(a,\rho))
 =\varnothing\big\}.
\end{equation}
Here \(\tau_{a,R}\le\tau_{z,L}\) follows from the second inclusion in
\eqref{eq:endpoint-ball-inclusions}.  Combining
\eqref{eq:endpoint-witness-z}--\eqref{eq:endpoint-event-inclusion-detailed}
and then taking the supremum over starts in \(\B_\delta(a,\rho/2)\) gives
\begin{equation}\Psi^{\mathrm{cen}}_{a;\rho,R}
       (\LE(\omega_i[0,\sigma]))\ge \mathbb P^z(\mathcal A_z)
 >B_M(h/L)^{\kappa_{\mathrm e,M}}
 =A_M^{\mathrm e}(\rho/R)^{\kappa_{\mathrm e,M}}.
 \label{eq:endpoint-forces-centered}
\end{equation}

Define the random-walk stopping time
\begin{multline*}
 \tau^{\mathrm{end}}_{i,a}(\xi)=\inf\big\{T\ge0:\
  \xi_T\in\B_\delta(a,R+\delta)\setminus\B_\delta(a,R),\\
  \supp\LE(\xi[0,T])\not\subset\B_\delta(a,R),\
  \Psi^{\mathrm{cen}}_{a;\rho,R}(\LE(\xi[0,T]))
     >A_M^{\mathrm e}(\rho/R)^{\kappa_{\mathrm e,M}}\big\}.
\end{multline*}
It is a stopping time, since every object tested at \(T\) is determined
by \(\xi[0,T]\).  Equations
\eqref{eq:first-exit-shell}, \eqref{eq:endpoint-last-exit-witness}, and
\eqref{eq:endpoint-forces-centered} show that
\begin{equation}\label{eq:endpoint-event-inclusion}
 \mathcal B_{i,a}
 \subset\{\tau^{\mathrm{end}}_{i,a}(\omega_i)
              \le\sigma<N_i\}.
\end{equation}

The event that defines \(\tau^{\mathrm{end}}_{i,a}\) is a subevent of
the event in \cref{clm:centered-prefix}.  Therefore
\begin{equation}\label{eq:endpoint-raw-bound}
 \mathbb P^a(\tau^{\mathrm{end}}_{i,a}<\tau_{D^\delta})
 \le C_M(\rho/R)^M
 \le C_M(h/L)^M,
\end{equation}
where \(16^M\) has been absorbed in \(C_M\).  At the stopped point
\(x\), \(|x-a|\le R+\delta\), and
\eqref{eq:endpoint-uniform-separation} gives
\begin{equation*}
 |x-b|\ge |a-b|-|x-a|
 \ge\frac12d_i^0-(R+\delta)\ge\frac14d_i^0.
\end{equation*}
Consequently \eqref{eq:green-3d-bounds} gives
\begin{equation}\label{eq:endpoint-Doob-factor}
 \sup_{x\in\B_\delta(a,R+\delta)}
 \frac{\mathcal G_{D^\delta}(x,b)}
      {\mathcal G_{D^\delta}(a,b)}\le C.
\end{equation}
In particular \(x\ne b\), so the correction term in
\eqref{eq:stopped-Doob-exact} vanishes.  Applying
\cref{clm:stopped-Green} to
\eqref{eq:endpoint-event-inclusion} and using 
\eqref{eq:endpoint-raw-bound} and~\eqref{eq:endpoint-Doob-factor} yields
\begin{equation}\label{eq:first-endpoint-bound}
 \mathbb P_{D^\delta}^{a,b}(\mathcal B_{i,a})
 \le C_M(h/L)^M.
\end{equation}

For the other endpoint, use the standard reversal identity for LERW
under a symmetric kernel; it follows directly from the one-path formula
in \cite[Propositions~3.1 and~3.5]{Lawler18}:
\begin{equation}\label{eq:LERW-reversal-law}
 \LE(\omega)^\leftarrow\text{ under }\mathbb P_{D^\delta}^{a,b}
 \quad\stackrel d=\quad
 \LE(\widetilde\omega)\text{ under }\mathbb P_{D^\delta}^{b,a}.
\end{equation}
Since $\mathcal B_{i,b}$ depends only on the trace, it is the the corresponding bad event near the starting
endpoint for the reversed path. Applying
\eqref{eq:first-endpoint-bound} with $a,b$ interchanged gives
$\mathbb P_{D^\delta}^{a,b}(\mathcal B_{i,b})\le C_M(h/L)^M$.
Together with the estimate for the starting endpoint, this proves the claim.
\end{proof}

\emph{Step 4: transfer from ordinary initial segments to the entangled
multipath LERW.}
For each vertex \(w\in C_r^\delta\), \cref{lem:prefix-witness} already
provides \(i,T\) with
\(\omega_{i,T}=w\) and
\(\supp\LE(\omega_i[0,T])\subset C_r^\delta\).
If \(w\ne b_i^\delta\), then \(T<N_i\).

Let \(\mathsf{Bad}\) be the union, over \(i\), of
\(\mathcal B_{i,a_i^\delta}\), \(\mathcal B_{i,b_i^\delta}\), and all
\(\{\tau_Q(\omega_i)<N_i\}\) with \(Q\) \(i\)-far in the sense of
\eqref{eq:i-far}.  By \eqref{eq:far-cube-bound},
\eqref{eq:endpoint-bound}, and \eqref{eq:cube-cover-count},
\begin{equation}\label{eq:bad-union}
 \mathbb P(\mathsf{Bad})
 \le C_{M,r}\left((h/L)^M+s^{-3}h^{-1}(s/h)^M\right).
\end{equation}
The next claim combines the off-center and endpoint estimates to control avoidance of the entire entangled range.

\begin{claim}
\label{clm:good-event-geometry}
On \(\mathsf{Bad}^c\),
\begin{equation*}
 \cH_\delta(C_r^\delta;s,L;U)
 \le C_M\max\{(s/h)^{\kappa_M},
                 (h/L)^{\kappa_{\mathrm e,M}}\}.
\end{equation*}
\end{claim}

\begin{proof}
Fix \(z\in U\cap\delta\Z^3\) within distance \(s\) of \(C_r^\delta\),
and choose a vertex \(w\in C_r^\delta\) with \(|z-w|\le s+\delta\).
By \cref{lem:prefix-witness}, some
\(\gamma=\LE(\omega_i[0,T])\) ends at \(w\) and satisfies
\(\supp\gamma\subset C_r^\delta\).  For the cube \(Q\) containing
\(w\), we have \(|z-v_Q|<3s\).

If \(Q\) is \(i\)-far, then \(w\ne b_i^\delta\), hence \(T<N_i\).
On \(\mathsf{Bad}^c\), the definition of the first bad time
\(\tau_Q\) therefore gives
\(\Psi_Q(\gamma)\le A_M(s/h)^{\kappa_M}\).
Since \(\Lambda_Q(\gamma)\subset C_r^\delta\) and
\(\B_\delta(v_Q,h/2)\subset\B_\delta(z,L)\), this bounds the
probability of avoiding \(C_r^\delta\) until \(\tau_{z,L}\).

Otherwise \(|z-e|<3h\) for an endpoint
\(e\in\{a_i^\delta,b_i^\delta\}\).  The complement of
\(\mathcal B_{i,e}\) bounds that probability by
\(B_M(h/L)^{\kappa_{\mathrm e,M}}\), because
\(\supp\eta_i\subset C_r^\delta\).
Taking the supremum over \(z\) proves the claim.
\end{proof}

\emph{Step 5: exponent bookkeeping and conclusion.}
We now make the earlier choice \(M=12\).  From
\eqref{eq:three-scales} and \eqref{eq:scale-ratios},
\begin{equation*}
 (h/L)^{12}=\eps^9,
\end{equation*}
whereas the number of cubes and the \(h^{-1}\) cost of the Doob
transform give
\[
 s^{-3}h^{-1}(s/h)^{12}=R_*^{-4}\eps^{-6}\eps^{-5/4}\eps^9=R_*^{-4}\eps^{7/4}.
\]
Since \(R_*\) is fixed, \eqref{eq:bad-union} becomes
\begin{equation}\label{eq:bad-final}
 \sup_{0<\delta<\delta_0(\eps)}
 \mathbb P(\mathsf{Bad})
 \le C(M,r,R_*)(\eps^9+\eps^{7/4})\longrightarrow0.
\end{equation}
On \(\mathsf{Bad}^c\), \cref{clm:good-event-geometry} and
\eqref{eq:scale-ratios} give
\[
 \cH_\delta(C_r^\delta;s,L;U)\le C\eps^\alpha,
\]
where
\begin{equation*}
 \alpha=\frac34\min\{\kappa_{12},\kappa_{\mathrm e,12}\}>0.
\end{equation*}
Set \(\eta_H=\alpha/2\).  There is \(\eps_0>0\), independent of
\(\delta\), such that \(C\eps^\alpha\le\eps^{\eta_H}\) whenever
\(\eps<\eps_0\).  Hence, for every such \(\eps\) and all
\(\delta<\delta_0(\eps)\),
\[
 \left\{\cH_\delta(C_r^\delta;\eps^2R_*,
             \sqrt\eps R_*;U)>\eps^{\eta_H}\right\}
 \subset\mathsf{Bad}.
\]
Take first \(\limsup_{\delta\downarrow0}\), use
\eqref{eq:bad-final}, and then let \(\eps\downarrow0\).  This is exactly
\eqref{eq:H2}.
\end{proof}

\section{Stability of loop decoration}\label{sec:decoration}
In this section, we use the hittability estimate of
\cref{prop:multipath-hittability} to obtain the near-miss estimate of
\cref{cor:no-near-miss}. This estimate in turns verifies the stability of loop decoration
under Hausdorff convergence in \cref{prop:common-decoration-stability},
allowing finite recovery identities such as \eqref{eq:finite-transform}
to pass to the limit through \cref{cor:recovery-transfer}. This provides the three-dimensional counterpart of~\cite[Section 3]{Cai26}.

The following arguments adapt the loop-soup coupling and recovery proof of
\cite[Theorem~2.2, Propositions~5.1--5.2 and Section~5.2]{SS18}.
The extension needed here is to arbitrary targets satisfying local
hittability, including targets that reach the boundary; the soup is
always independent of the target. We record the common statement once,
for use with both entangled multipath LERW and loop-catchers.

We first couple the discrete and Brownian soups so that their macroscopic loops can be matched.
\begin{lemma}
\label{lem:macro-soup-coupling}
Let \(D\subset\R^3\) be a bounded connected domain with Lipschitz boundary.  Fix
\(0<t_0<T<\infty\) and \(a>0\).  There is a coupling of
the rescaled random-walk loop soup and Brownian loop soup such that, with
probability tending to one as \(\delta\downarrow0\):
\begin{enumerate}[label=\textnormal{(\alph*)},leftmargin=2.4em]
\item every loop contained in the corresponding domain, with duration in
  \([t_0,T]\) and diameter at least \(a\), has a unique partner;
\item paired traces have Hausdorff distance at most
  \(\gamma_\delta\), where \(\gamma_\delta\to0\);
\item for each matched pair considered above, the discrete
loop lies in $D^\delta$ if and only if the Brownian loop
lies in $D$, except on an event whose probability tends
to zero.
\end{enumerate}
The number of loops in the stated window is almost surely finite.  The
coupling may be chosen independently of any prescribed random target.
\end{lemma}

\begin{proof}
Apply \cite[Theorem~2.2]{SS18} with mesh \(\delta\) and, for example,
its exponent \(\theta=1\). On a fixed positive duration window, it
pairs the full-space loops with vanishing duration and uniform spatial
errors. Restriction to \(D\) follows from
\cite[Proposition~5.1]{SS18}: macroscopic discrete loops contained in
\(D^\delta\) but approaching \(\partial D\) have vanishing mass.
The extension to Lipschitz domains follows from \eqref{eq:boundary-escape}, as explained in \cref{sec:domain-extension}.
For Brownian loops the corresponding assertion follows by continuity
from above, since the window has finite mass and every loop in \(D\)
is compactly contained there. These observations remove mismatched
pairs at the boundary. The duration measure has no atoms, so shifting
the duration cutoffs by the coupling error is harmless. Local finiteness
and independence from any target are part of this Poisson construction.
\end{proof}

For a discrete loop \(\ell\), let \(t_\delta(\ell)\) be its diffusive
duration.  Given \(0<t_0<T<\infty\), write
\(m_{D^\delta}^{\delta,[t_0,T]}\) for random-walk loop measure
restricted to this duration window, and
\(\mu_D^{\mathrm{loop},[t_0,T]}\) for its Brownian counterpart.

The next lemma bounds the mass of loops that approach a hittable target without hitting it.
\begin{lemma}\label{lem:near-miss}
Let \(U_0\Subset U_1\Subset D\).  There is a finite constant \(C\), depending
only on \(t_0,T,U_0,U_1,D\), with the following property.  Suppose a
random lattice set \(C^\delta\subset U_0\) is independent of the test loop,
and define the target-measurable event
\begin{equation}\label{eq:H-event}
 E^\delta:=\left\{
 \sup_{\substack{z\in U_1\cap\delta\Z^3\\
                  \dist(z,C^\delta)\le 2s}}
 \mathbb P^z\!\left(S[0,\tau_{z,L}]\cap C^\delta=\varnothing\right)
 \le q\right\}.
\end{equation}
Thus \(E^\delta\in\sigma(C^\delta)\).  Assume \(0<s<L\),
\(U_0^{+3L}\Subset U_1\), and \(\delta=o(s)\).  Then, on \(E^\delta\),
\begin{equation}\label{eq:near-miss-mass}
 m_{D^\delta}^{\delta,[t_0,T]}
 \{\ell:\dist(\tr\ell,C^\delta)\le s,
           \ \tr\ell\cap C^\delta=\varnothing\}
 \le C\left(q+\frac{L^2}{t_0}\right)+o_\delta(1).
\end{equation}
The same bound holds after imposing any further trace-measurable
restriction on the loops.
It also holds for an arbitrary target \(C^\delta\subset D^\delta\)
if only loops whose entire trace lies in \(U_0\) are counted, with the
same hittability event \(E^\delta\).
\end{lemma}

\begin{proof}
This is the rooted-loop argument of
\cite[proof of Proposition~5.2]{SS18}, with its LERW hittability input
replaced by \(E^\delta\). Condition on the target and keep the original
root and the factor \(1/n\) in loop measure. By time reversal, at a
cost of two we may assume the first approach occurs before \(n/2\).
After that approach, avoidance until exiting the ball of radius \(L\)
costs at most \(q\); failure to exit within \(n/4\) steps costs at most
\(CL^2/t_0\). In either case, a final interval of at least \(n/4\)
steps gives a return factor \(Cn^{-3/2}\). Thus the probability for a
loop rooted at any fixed vertex is at most
\[
 C\left(q+L^2/t_0\right)n^{-3/2}.
\]
Summing with weight \(1/n\) over \(O(\delta^{-3})\) roots and
\(3t_0\delta^{-2}\le n\le3T\delta^{-2}\) gives
\eqref{eq:near-miss-mass}. The condition \(\delta=o(s)\) allows us to
pass from the interpolated trace to a vertex within distance \(2s\).
For the localized version, the approach point lies in \(U_0\), which
is the only place where hittability is used.
\end{proof}

Whenever \(C_r^\delta\) appears below, it is the linearly interpolated trace
of the entangled multipath LERW constructed from a fixed finite family of
independent normalized random-walk bridges between lattice approximations
of pairs \((a_i,b_i)\) with \(a_i\ne b_i\); repeated pairs are allowed.
We write
\begin{equation*}
 C^{0,\delta}=\bigcup_{i=1}^r\tr(\omega_i^\delta)
\end{equation*}
for the union of those bridges.  The endpoint and scale setup is given in \cref{sec:hittability-input}.

The next corollary says that removing the duration and boundary cutoffs excludes macroscopic loops that approach a hittable target without hitting it.
\begin{corollary}
\label{cor:no-near-miss}
Let \(A^\delta\subset D^\delta\) be random vertex sets. Assume that, for
every \(U\Subset D\) and all sufficiently small fixed \(R_*>0\), there
is \(\alpha>0\) such that
\begin{equation}\label{eq:local-target-hittability}
 \lim_{\eps\downarrow0}\limsup_{\delta\downarrow0}
 \mathbb P\!\left(
 \cH_\delta(A^\delta;\eps^2R_*,\sqrt\eps R_*;U)
       >\eps^\alpha\right)=0.
\end{equation}
Then, for every \(a>0\), an independent intensity-one random-walk loop
soup satisfies
\begin{equation}\label{eq:common-close-miss}
 \lim_{s\downarrow0}\limsup_{\delta\downarrow0}
 \mathbb P\big(\exists\ell\in\mathcal L_{D^\delta}^{1,\delta}:
       \diam\ell\ge a,\ \dist(\tr\ell,A^\delta)\le s,
       \supp\ell\cap A^\delta=\varnothing\big)=0.
\end{equation}
In particular, this applies to the entangled multipath LERW by
\cref{prop:multipath-hittability}. 
\end{corollary}

\begin{proof}
Fix the diameter cutoff \(a\). By \cite[Proposition~5.1 and its proof]{SS18},
the mass \(q_\delta(\eta)\) of loops contained in \(D^\delta\), of diameter
at least \(a\), and within distance \(\eta\) of \(\partial D\) satisfies
\begin{equation}\label{eq:blc-radial-collar-mass}
 \lim_{\eta\downarrow0}\limsup_{\delta\downarrow0}q_\delta(\eta)=0.
\end{equation}
For the present Lipschitz domains, use the boundary estimate of \cref{sec:domain-extension}. This lets us restrict all approach points to
a fixed compact subset of \(D\), even when the target reaches the boundary.

On a fixed duration window \([t_0,T]\), apply the localized version of
\cref{lem:near-miss} with
\(s=\tfrac12\eps^2R_*\), \(L=\sqrt\eps R_*\) and
\(q=\eps^\alpha\). Outside the exceptional event in
\eqref{eq:local-target-hittability}, the near-miss mass is at most
\(C(\eps^\alpha+\eps R_*^2/t_0)+o_\delta(1)\).
The Poisson formula, followed by \(\delta\downarrow0\),
\(\eps\downarrow0\) and \(\eta\downarrow0\), proves the claim in that window.

Finally, the Gaussian and spectral bounds \eqref{eq:lclt}--\eqref{eq:discrete-spectral-bound}
show that the masses of loops of diameter at least $a$ and duration below \(t_0\),
and of loops of duration above \(T\), are bounded respectively by
\[
 C\int_0^{t_0}t^{-5/2}e^{-ca^2/t}\dd t,
 \qquad C\int_T^\infty t^{-1}e^{-ct}\dd t.
\]
These bounds hold for both soups and vanish as
\(t_0\downarrow0\), \(T\uparrow\infty\), completing the proof.
\end{proof}

For \(\lambda\in[0,1]\), write
\(\mathcal L_{D^\delta}^{\lambda,\delta}\) and \(\mathcal L_D^\lambda\)
for the discrete and Brownian soups. Each soup is independent of the
targets; the two soups may be coupled to each other. We use
\(\mathcal D_\lambda\) from \eqref{eq:decoration} in the domain \(D\).
Its discrete counterpart \(\mathcal D_\lambda^\delta\) attaches loops
whose vertex supports meet the target support and then interpolates the
resulting set. These conventions apply equally to linearly interpolated
traces and to cell traces, which differ by \(O(\delta)\).
As we are going to see in the next proposition, the absence of such near misses makes the selection of attached loops stable under Hausdorff convergence.
\begin{proposition}
\label{prop:common-decoration-stability}
Let \(D\) be a bounded connected domain with Lipschitz boundary. Suppose
that random vertex sets
\(A^\delta\subset D^\delta\) converge almost surely in Hausdorff distance
to a random compact \(A\subset\overline D\). 
Assume the near-miss estimate \eqref{eq:common-close-miss}.
For every fixed \(\lambda\in[0,1]\), attach the loops of an independent
intensity-\(\lambda\) soup that directly hit the target, and take the
closure. Then
\[
 \big(A^\delta,\mathcal D_\lambda^\delta(A^\delta)\big)
 \Rightarrow\big(A,\mathcal D_\lambda(A)\big).
\]
The same conclusion holds jointly with additional coordinates that
converge in the same realization, provided the loop soup is independent
of the full jointly convergent tuple. Discrete hits in this statement use vertex
supports; interpolation is performed afterwards.
\end{proposition}

\begin{proof}
We use the selection argument of \cite[Section~5.2]{SS18} with
\eqref{eq:common-close-miss} in place of its LERW estimate.
Couple the soups independently of the jointly convergent target tuple,
using \cref{lem:macro-soup-coupling}, and first keep only loops in a
fixed duration window with diameter at least a continuity value \(a>0\).
Common thinning gives the coupling at intensity \(\lambda\).

If both the target and loop coupling errors are below \(\rho\), a
Brownian hit of \(A\) produces a discrete loop within \(2\rho\) of
\(A^\delta\), which hits by \eqref{eq:common-close-miss} except with
vanishing probability. Conversely, a discrete hit produces a Brownian
hit of \(A^{+2\rho}\). Hence, within the finite window,
\begin{equation}\label{eq:selection-sandwich}
 \{\ell:\ell\cap A\ne\varnothing\}
 \subset\operatorname{pair}
       \{\ell_\delta:\supp\ell_\delta\cap A^\delta\ne\varnothing\}
 \subset\{\ell:\ell\cap A^{+2\rho}\ne\varnothing\},
\end{equation}
with error tending to zero as first \(\delta\downarrow0\) and then
\(\rho\downarrow0\). Conditional on \(A\), the mass of loops in the last
family but not the first decreases to zero: disjoint compact sets have
positive distance and the window measure is finite. Thus the selected
finite families agree with probability tending to one.

Remove the duration cutoffs using the bounds in the proof of
\cref{cor:no-near-miss}. Every remaining attached loop of diameter
less than \(a\) lies within distance \(a\) of its target, so deleting
these loops changes the decorated set by at most \(a+O(\delta)\) in
Hausdorff distance. Let \(a\downarrow0\). The same coupling keeps the
additional coordinates, proving joint convergence as well.
\end{proof}

Consequently, a finite recovery property such as \eqref{eq:finite-transform} passes to a subsequential continuum limit in the form \eqref{eq:common-recovery-transfer}. The argument is similar to~\cite[Proposition~4.2]{Cai26}, and we omit it here.
\begin{corollary}
\label{cor:recovery-transfer}
Assume the hypotheses of \cref{prop:common-decoration-stability}, and
fix \(\lambda\in[0,1]\). Let random vertex sets \(R^\delta\) have
geometric traces converging in law to a compact set \(R\). Suppose that,
for every \(F\subset D^\delta\),
\[
 \mathbb E\!\left[
 \one_{\{A^\delta\cap F=\varnothing\}}
 e^{-\lambda m_{D^\delta}^{\delta}(A^\delta,F)}\right]
 =\mathbb P(R^\delta\cap F=\varnothing).
\]
Then \(\mathcal D_\lambda(A)\stackrel d=R\), and, for every compact
\(E\subset\overline D\),
\begin{equation}\label{eq:common-recovery-transfer}
 \mathbb E\!\left[
 \one_{\{A\cap E=\varnothing\}}e^{-\lambda m_D(A,E)}\right]
 =\mathbb P(R\cap E=\varnothing).
\end{equation}
\end{corollary}

\section{Characterization of 3D LEBM}
\label{sec:continuum-probe}

Following~\cite[Section 4.1]{Cai26}, we construct the intensity-one continuum probes and apply the common
identification results of \cref{sec:prelim} to prove \cref{thm:main}.

Let \(D\subset\R^3\) be a bounded Lipschitz domain, in particular the domain \(B\) fixed in \cref{sec:intro}.  Fix ordered pairs of
distinct interior points \((x_i,y_i)_{i\le r}\), with repetitions allowed,
and lattice approximations \((x_i^\delta,y_i^\delta)\).  Let
\(C_r^\delta\) be the linearly interpolated range of the entangled multipath LERW and
\begin{equation*}
 C^{0,\delta}=\bigcup_{i=1}^r\tr(\omega_i^\delta)
\end{equation*}
the union of the normalized random-walk bridges.  By construction,
\(C_r^\delta\subset C^{0,\delta}\).

The compactness of \(\K(J_D)\), together with tightness of the path
coordinates, allows us to pass to a subsequence on which

\begin{equation}\label{eq:joint-core-limit}
(C_r^\delta,C^{0,\delta},\omega_1^\delta,\ldots,\omega_r^\delta)
 \Rightarrow(C,C^0,\Gamma_D^{x_1,y_1},\ldots,\Gamma_D^{x_r,y_r}),
 \qquad C\subset C^0\quad\text{a.s.}
\end{equation}
The containment follows from its discrete counterpart and closedness of
the containment relation in the product Hausdorff topology.  Any such
subsequential \(C\) will be called a \emph{continuum interior probe}.
As explained in~\cite{Cai26}, the law of \(C\) need not be identified separately.
The following proposition shows that every subsequential interior probe has the required Green-function transform.
\begin{proposition}
\label{prop:continuum-probe}
Let \(C\) be any subsequential limit in \eqref{eq:joint-core-limit}.  For
every compact \(A\subset\overline D\),
\begin{equation}\label{eq:continuum-probe-identity}
 \mathbb E_C\!\left[
   \one_{\{C\cap A=\varnothing\}}e^{-m_D(C,A)}
 \right]
 =\prod_{i=1}^r\frac{G_{D\setminus A}(x_i,y_i)}{G_D(x_i,y_i)},
\end{equation}
where the numerator is zero if an endpoint lies in \(A\) or the endpoints
are in different components of \(D\setminus A\).
\end{proposition}

We omit the proof because \cref{cor:recovery-transfer}, applied to
Cai's finite identity \eqref{eq:finite-transform} using
\cref{lem:green-path-limit,prop:multipath-hittability,cor:no-near-miss},
gives the three-dimensional counterpart of
\cite[Proposition~4.2]{Cai26}.

\begin{proof}[Proof of \cref{thm:main}]
For \(\lambda=1\), \cref{prop:continuum-probe} supplies the probes
required by \cref{lem:mixed-moments}. Every law on \(\Gamma\)
with the recovery property satisfies
\eqref{eq:shared-recovery-transform}, by
Fact \ref{lem:closed-decoration-avoidance}. Since
\(\Gamma\subset\mathfrak C_B\), the shared identification lemma
proves uniqueness.

Compactness in the ambient space \(\K(J_B)\) supplies subsequential
limits of the rescaled LERW traces. Each is contained in \(\overline B\),
is supported on \(\Gamma\), and satisfies \eqref{eq:lerw-recovery}
by the domain extension of \cite[Theorems~1.1--1.2]{SS18} in
\cref{sec:domain-extension}. The uniqueness just proved identifies
all such limits and hence gives convergence of the full sequence.
\end{proof}

\section{Construction and characterization of the 3D Brownian loop-catcher}
\label{sec:blc-discrete}\label{sec:blc-uniqueness}

We prove \cref{thm:blc-main} using the finite recovery identities
\eqref{eq:blc-discrete-recovery-transform}--\eqref{eq:blc-discrete-probe-transform}. The additional input is the
inclusion of loop-erased initial segments of the recovered walks in the
catcher, which transfers the hittability estimate uniformly in
\(\lambda\). Continuum recovery and uniqueness then follow from the
common results in \cref{sec:prelim,sec:decoration}, as for LERW.

\subsection{Loop-erased initial segments in recovered walks}
Recall the catcher and recovered walks defined in
\cref{sec:discrete-model}.
As we are going to see in the lemma below, the recovery algorithm places a loop-erased initial segment through every catcher vertex inside the catcher.

\begin{lemma}
\label{lem:blc-online-witness}
We suppress the mesh superscript in this pathwise statement.
In the coupling above, let \(\widehat\omega_i\) be the recovered paths
and let \(S_{\lambda,\times}\) be the range of the time-parametrized
multipath loop-catcher.  For every \(w\in S_{\lambda,\times}\), there
are an index \(i\) and a time \(\sigma\) such that
\begin{equation}\label{eq:blc-online-witness}
 \widehat\omega_i(\sigma)=w,
 \qquad
 \supp\LE(\widehat\omega_i[0,\sigma])
       \subset S_{\lambda,\times}.
\end{equation}
Moreover, the loop erasure of each complete recovered path is contained
in \(S_{\lambda,\times}\).  The same assertion holds for the one-path
excursion catcher.
\end{lemma}

\begin{proof}
Use the layers of \cite[Algorithms~2.5 and~2.7]{Cai26}: let
\(R_0\) be the initial vertex domain, \(R_n\) the domain remaining
at layer \(n\), and \(r_n\) the number of ordered entangled paths
\(\gamma_1^{(n)},\ldots,\gamma_{r_n}^{(n)}\) in that layer.
Write \(C_{n,i}=\bigcup_{j\le i}\supp\gamma_j^{(n)}\), with
\(C_{n,0}=\varnothing\), and
\(S_n=C_{n,r_n}\), so that \(R_{n+1}=R_n\setminus S_n\) and
\(S_{\lambda,\times}=\bigcup_n S_n\).  At an occurrence of a layer path,
read the recovered path after all blocks attached there have returned.
The chronological stack then agrees with that obtained from the prefix of the layer path and its actual incoming stack (Here the stack is the chronological loop erasure of the
recovered path read so far, and for a layer path, its
incoming stack is the stack present when that layer
path starts to be processed). Indeed, by \cite[Lemma~2.3]{Cai26}, a
block based at \(u\) avoids the earlier vertices of that layer path and earlier paths
of its layer; recovery preserves this avoidance because its child
substitutions lie in \(R_{n+1}\), and assigned loops obey the same
avoidance condition.  Such a block therefore misses the incoming stack
except at its top and restores it on returning to \(u\).  Let \(T\) be the vertex support of the incoming chronological stack. The needed incoming condition, \(T\cap R_n\subset C_{n,i-1}\), holds at every
recursive entry: it holds for the empty top-level stack, and processing
the children in their recorded order adds inside \(R_{n+1}\) only vertices
of earlier child paths; the connecting path segments between successive
recursive calls lie in \(S_n\).
Induction over layers and this order proves the condition and the layer-path
stack assertion.  Every catcher vertex occurs on a layer path, where the
global stack is contained in
\((R_0\setminus R_n)\cup C_{n,i-1}\cup\supp\gamma_i^{(n)}
\subset S_{\lambda,\times}\), proving the required inclusion.  At the terminal
top-level occurrence the stack is contained in its layer path; this proves
the assertion about the loop erasure of the complete recovered path as well.
\end{proof}

\subsection{A uniform hittability estimate}
\label{sec:blc-hittability}

The benefit of \cref{lem:blc-online-witness} is that all estimates can be
made on the recovered walks.  Their law does not depend on \(\lambda\).

\begin{proposition}
\label{prop:blc-uniform-hittability}
Fix \(U\Subset B\) and
\(0<R_*<\dist(U,\partial B)/20\).  There is \(\eta>0\) such that
\begin{equation}\label{eq:blc-radial-hittability}
 \lim_{\eps\downarrow0}\limsup_{\delta\downarrow0}
 \sup_{0\le\lambda\le1}
 \mathbb P\!\left(
  \cH_\delta(S_\lambda^\delta;\eps^2R_*,\sqrt\eps R_*;U)
       >\eps^\eta\right)=0.
\end{equation}
For each fixed finite endpoint list, the same assertion holds with
\(S_\lambda^\delta\) replaced by \(S_{\lambda,\times}^\delta\).
\end{proposition}

\begin{proof}
Apply the proof of \cref{prop:multipath-hittability} to the recovered
paths supplied by \cite[Proposition~2.8]{Cai26}. Then by
\cref{lem:blc-online-witness}, 
we find
\[
 \left\{\cH_\delta(S_{\lambda,\times}^\delta;
       \eps^2R_*,\sqrt\eps R_*;U)>\eps^\eta\right\}
 \subset\left\{\widehat{\boldsymbol\omega}^{\,\lambda}
                         \in\mathsf{Bad}_{\delta,\eps}\right\}.
\]
The recovered tuple has the same normalized random-walk bridge law for
every \(\lambda\), so \eqref{eq:bad-final} gives a bound tending to zero
uniformly in \(\lambda\).  In the radial case the recovered path is an
unconditioned killed walk: the same proof no longer needs a Doob-transform factor and uses only the estimate near the starting
point, and its terminal point lies
outside a fixed neighbourhood of \(U\) for small \(\delta\).
These smaller bounds prove the radial assertion.  The uniformity is in
marginal probabilities and requires no joint coupling of all intensities.
\end{proof}

\subsection{Subsequential limits}

\begin{proposition}
\label{prop:blc-subsequential-geometry}
For fixed \(\lambda\in[0,1]\), the family \((K_\lambda^\delta,R^\delta)\)
is tight.  Along every convergent subsequence it has a Skorokhod realization
such that
\begin{equation*}
 (K_\lambda^\delta,R^\delta)\longrightarrow(K,R)
 \quad\text{a.s.},
 \qquad
 K\subset R\stackrel d=W[0,\tau_B].
\end{equation*}
Here \(K\) is compact and connected, \(0\in K\),
\(K\cap\partial B\) is a singleton, and \(\operatorname{int}K=\varnothing\).

For a fixed finite endpoint list, every mesh sequence has a subsequence on

which \(C_{\lambda,\times}^\delta\Rightarrow C\), jointly with
mutually independent normalized Brownian bridges
\((\Gamma_B^{a_i,b_i})_{i\le r}\), and

\begin{equation}\label{eq:blc-probe-subset}
 C\subset\bigcup_{i=1}^r\tr(\Gamma_B^{a_i,b_i})\Subset B.
\end{equation}
The limit has at most \(r\) connected components, and every
component contains both endpoints of at least one labelled pair.
\end{proposition}

We omit the proof since it is similar to \cite[proof of Proposition~4.2]{Cai26} as well.

\subsection{Continuum recovery identities}

We now apply the same recovery transfer as for the LERW probes.
The local hittability hypothesis in \cref{cor:no-near-miss} was formulated
to allow targets meeting the boundary, so the radial case requires no
separate loop-selection proof here.
In the proposition below, the common recovery-transfer result gives both the stopped Brownian recovery law \eqref{eq:blc-continuum-recovery} and the continuum probe identities \eqref{eq:blc-continuum-probe} for loop-catchers.
\begin{proposition}
\label{prop:blc-continuum-identities}
Every radial subsequential limit \(K\) in
\cref{prop:blc-subsequential-geometry} satisfies, for every compact
\(E\Subset B\),
\begin{equation}\label{eq:blc-continuum-recovery}
 \mathbb E\!\left[
  \one_{\{K\cap E=\varnothing\}}e^{-\lambda m_B(K,E)}\right]
 =\mathbb P^0(W[0,\tau_B]\cap E=\varnothing).
\end{equation}
For every finite ordered list \((a_i,b_i)_{i\le r}\) with
\(a_i,b_i\in B\), \(a_i\ne b_i\), and repeated pairs allowed, there is a
subsequential multipath limit \(C=C_\lambda(a_1,b_1;\ldots;a_r,b_r)\)
satisfying, for every compact \(E\subset\overline B\),
\begin{equation}\label{eq:blc-continuum-probe}
 \mathbb E_C\!\left[
  \one_{\{C\cap E=\varnothing\}}e^{-\lambda m_B(C,E)}\right]
 =\prod_{i=1}^r\frac{G_{B\setminus E}(a_i,b_i)}{G_B(a_i,b_i)}.
\end{equation}
Here a numerator is zero if an endpoint lies in \(E\) or the two endpoints
belong to different components of \(B\setminus E\).
\end{proposition}

We omit the proof, since \cref{cor:recovery-transfer} applies to
\eqref{eq:blc-discrete-recovery-transform}--\eqref{eq:blc-discrete-probe-transform}
by \cref{prop:blc-uniform-hittability,cor:no-near-miss,prop:blc-subsequential-geometry},
exactly as for \cref{prop:continuum-probe}.

\subsection{Uniqueness and the full scaling limit}

\begin{proof}[Proof of \cref{thm:blc-main}]
Fix \(0<\lambda\le1\). By
\cref{prop:blc-subsequential-geometry}, every mesh sequence has a
subsequence converging to a law on \(\mathfrak C_B\), and
\cref{prop:blc-continuum-identities} gives its recovery transform
\eqref{eq:blc-continuum-recovery}. The same proposition supplies
the probes in \eqref{eq:shared-probe-hypothesis} at intensity
\(\lambda\), thus \cref{lem:mixed-moments} identifies all
subsequential laws.
The recovery property \eqref{eq:blc-recovery-intro} is given by 
\eqref{eq:blc-continuum-recovery}.
\end{proof}

\appendix
\section{Inversion invariance of whole-space LEBM}
\label{sec:whole-space}
In this appendix, we sketch a proof of \cref{cor:inversion}.
Write \(B_R=B(0,R)\), let \(m_\infty\) be Brownian loop measure in
\(\R^3\), and let \(\mathcal L_\infty^1\) be its intensity-one soup.
For sets \(A,C\subset\widehat{\R^3}\), write \(m_\infty(A,C)\) for
the mass of loops meeting both \(A\cap\R^3\) and \(C\cap\R^3\).
For an infinite simple path \(K\) from \(0\) to \(\infty\), define
\(\mathcal D_\infty(K)\) as in \eqref{eq:decoration}, using
\(\mathcal L_\infty^1\) and taking closure in \(\widehat{\R^3}\).
Here ``infinite'' means that the path eventually leaves every bounded set;
equivalently, adjoining \(\infty\) gives a simple compact path in $\widehat{\R^3}$.
All soups below are independent of the paths they decorate.

The following lemma gives a whole-space version of Theorem \ref{thm:main}.
\begin{lemma}\label{lem:whole-space}
The law of \(K_\infty\) is the unique law on infinite curves from
\(0\) to \(\infty\) such that
\begin{equation}\label{eq:whole-space-recovery}
 \mathcal D_\infty(K)\stackrel d=W[0,\infty)\cup\{\infty\},
\end{equation}
where \(W\) is Brownian motion started at \(0\).
\end{lemma}
\begin{proof}
We record the exhaustion details so that no whole-space uniqueness
statement is being assumed. We use the crossing-loop estimate
\cite[Lemma~5.1]{JegoLupu26}:
\begin{equation}\label{eq:whole-space-loop-tail}
 m_\infty(\overline B_L,B_R^c)\le cL/R,\qquad R\ge2L.
\end{equation}

First we prove uniqueness using only the bounded-domain probes already
constructed in \cref{prop:continuum-probe}. Fix finitely many pairs of distinct endpoints \((x_i,y_i)\), with repetitions allowed, and take the
corresponding probe \(C_R\) in \(B_R\). For two candidate laws
\(\mu_1,\mu_2\), set
\[
 X^R_{x,y}(K)=\frac{G_{B_R\setminus K}(x,y)}{G_{B_R}(x,y)},
 \qquad
 F^\infty_j(A)=
 \mathbb E_{\mu_j}\bigl[\one_{\{K\cap A=\varnothing\}}
                         e^{-m_\infty(K,A)}\bigr].
\]
The closed-decoration avoidance argument of
Fact \ref{lem:closed-decoration-avoidance}, localized using
\eqref{eq:whole-space-loop-tail}, applies to every compact
\(A\subset\R^3\). Hence \eqref{eq:whole-space-recovery} implies
\(F^\infty_1(A)=F^\infty_2(A)\).
Apply \cref{prop:continuum-probe} with the compact obstacle
\(K\cap\overline B_R\) and use Fubini as in
\cref{lem:mixed-moments}. On \(\{C_R\subset B_L\}\), the replacement
of \(m_{B_R}(K,C_R)\) by \(m_\infty(K,C_R)\) changes the integrand by
at most \(cL/R\). Therefore
\[
 \left|\mathbb E_{\mu_1}\prod_iX^R_{x_i,y_i}(K)
       -\mathbb E_{\mu_2}\prod_iX^R_{x_i,y_i}(K)\right|
 \le 2cL/R+2\mathbb P(C_R\not\subset B_L).
\]
The probe is contained in the union of its normalized Brownian
bridges. Each bridge in \(B_R\) is a whole-space normalized bridge
conditioned to stay in \(B_R\). Such whole-space bridges have finite
lifetime and compact trace, since \(G_{\R^3}(x,y)<\infty\).
Consequently the last probability tends to zero when first
\(R\to\infty\) and then \(L\to\infty\).
Monotone convergence of Green functions and bounded convergence
now give equality of all mixed moments of
\(G_{\R^3\setminus K}(x,y)/G_{\R^3}(x,y)\).
The local nonpolarity argument in
\cref{lem:no-polar-subarc} and the Green-coordinate argument cited before
\cref{prop:green-injective} apply to infinite curves: the required
local finiteness of crossing loops follows from
\eqref{eq:whole-space-loop-tail}. The condition of local
nonpolarity is again Borel, using rational balls in \(\R^3\).
The moment-determinacy and measurable-reconstruction arguments of
\cref{lem:mixed-moments,rem:borel-inverse} therefore identify the
law of the infinite curve. This proves uniqueness.

For existence, let \(K^{(R)}\) have the radial LERW scaling-limit law
in \(B_R\). The finite/infinite LERW comparison
\cite[(10.2), proof of Theorem~10.1]{LiShiraishi}
couples initial segments through scale \(S\) with error \(O(S/R)\).
After first leaving that scale, either LERW can return to \(B_L\)
only if its underlying random walk does so; the standard
three-dimensional return bound is \(O(L/S)\).
By \cite[Proposition~10.2]{LiShiraishi} and lower semicontinuity of
total variation under weak convergence, the comparison passes to the
limiting stopped curves. The return estimate passes to the limit first
with slightly enlarged observation balls. Outside these return events,
the full traces in \(B_L\) agree with their stopped traces. Thus
\[
 d_{\rm TV}\bigl(\mathop{\rm Law}(K^{(R)}\cap B_L),
                 \mathop{\rm Law}(K_\infty\cap B_L)\bigr)
 \le c(S/R+L/S),\qquad L\ll S\ll R.
\]
The comparison in the cited source uses cubes for the intermediate
scale; enclosing balls give the displayed bound with adjusted constants.
In particular the local trace laws converge in total variation.
The same comparison makes every stopped initial segment of
\(K_\infty\) simple. Applying the return bound at increasing dyadic
scales shows that it eventually leaves every bounded set. Thus adjoining
\(\infty\) gives a simple compact path in $\widehat{\R^3}$, as required by the candidate class.
Couple these traces on \(B_S\), independently of one whole-space soup.
On agreement there, their decorations inside \(B_L\), \(L<S\),
can differ only through a loop meeting both \(B_L\) and \(B_S^c\);
its probability is at most \(cL/S\) by
\eqref{eq:whole-space-loop-tail}.
But the \(B_R\)-decoration of \(K^{(R)}\) has the law of
\(W[0,\tau_{B_R}]\), by \cref{thm:main} and dilation.
Letting \(R\to\infty\) and then \(S\to\infty\) proves
\eqref{eq:whole-space-recovery} on every bounded observation set.
The endpoint at infinity is deterministic, completing the proof.
\end{proof}

\begin{proof}[Proof of \cref{cor:inversion}]
The Brownian inversion and time-reversal identities
\cite{Yor85}, equivalently
\cite[Lemmas~3.3 and~3.5]{JegoLupu26}, give
\begin{equation}\label{eq:inverted-brownian-range}
 \iota\bigl(W[0,\infty)\cup\{\infty\}\bigr)
 \stackrel d=W[0,\infty)\cup\{\infty\}.
\end{equation}
For clarity concerning the endpoints, apply those identities first
to Brownian motion from \(x\ne0\) to infinity. Its inverted and
reversed trace is a normalized Brownian bridge from \(0\) to
\(\iota(x)\). Writing \(y=\iota(x)\), the bridge stopped upon leaving
\(B_S\) has density
\[
 \frac{G_{\R^3}(W_{\tau_{B_S}},y)}{G_{\R^3}(0,y)}
 =1+O(S/|y|)
\]
relative to stopped Brownian motion, uniformly for \(|y|>2S\).
The same Doob-transform formula bounds its subsequent return probability
to \(B_L\) by \(CL/S\). Thus its compactified trace converges to the
Brownian trace as \(|y|\to\infty\). On the other side, Brownian traces
started at \(x\to0\) converge by translation, and inversion is continuous
on the compactification. This proves
\eqref{eq:inverted-brownian-range}.

By \cite[Lemma~4.1]{JegoLupu26},
\(\iota(\mathcal L_\infty^1)\) is again an intensity-one Brownian
loop soup and remains independent of \(\iota(K_\infty)\).
(Their inversion has an additional minus sign, removed by reflection
invariance.) Since \(\iota\) is a homeomorphism of
\(\widehat{\R^3}\), it preserves intersections and commutes with
closure. Thus
\[
 \mathcal D_\infty\bigl(\iota(K_\infty)\bigr)
 \stackrel d=\iota\bigl(\mathcal D_\infty(K_\infty)\bigr)
 \stackrel d=W[0,\infty)\cup\{\infty\}.
\]
The inverted path, with reversed orientation, is an infinite simple curve
from \(0\) to \(\infty\). Apply \cref{lem:whole-space}.
\end{proof}

\begin{remark}[Bounded-domain versions]\label{rem:bdeddomain}
Let $D\subset\mathbb R^3$ be a bounded connected Lipschitz domain.
All laws below concern unparametrized traces, and all decorations use an
independent intensity-one Brownian loop soup in $D$.

\smallskip
\noindent\emph{Interior point-to-point paths.}
For distinct $a,b\in D$, let $K_D^{a,b}$ denote the scaling limit of the
loop erasure of the normalized random-walk bridge from $a^\delta$ to
$b^\delta$ in $D^\delta$. It is a simple curve compactly contained in $D$,
and its law is uniquely characterized, among such curves with endpoints
$a,b$, by
\[
  \mathcal D_1^D(K_D^{a,b})
  \overset{d}=\operatorname{tr}(\Gamma_D^{a,b}).
\]

\smallskip
\noindent\emph{Exit excursions.}
Let $K_D^a$ have the radial law of \cref{thm:main}, with starting point $a$,
and write $\zeta(K)$ for the unique point of $K\cap\partial D$.
For a continuous $f:\partial D\to[0,\infty)$ with
\[
  Z_D^{a,f}:=\mathbb E^a[f(W_{\tau_D})]>0,
\]
define the law of $K_D^{a;f}$ by
\[
 \frac{d\operatorname{Law}(K_D^{a;f})}
      {d\operatorname{Law}(K_D^a)}(K)
 =\frac{f(\zeta(K))}{Z_D^{a,f}}.
\]
Its decoration has the law of the stopped Brownian range weighted by
$f(W_{\tau_D})/Z_D^{a,f}$, and this recovery property uniquely
characterizes its law among simple curves from $a$ with a single boundary
contact. Equivalently, it is the scaling limit obtained by weighting the
discrete exit excursion by $f$ of its exit point. 
The case
$f\equiv1$ is the unconditioned exit excursion. Disintegration by the
common exit point also gives the fixed-exit-point characterization for
harmonic-measure-almost every $z\in\partial D$; it does not, by itself,
assert convergence for arbitrary lattice approximations of a prescribed
boundary point.

\smallskip
\noindent\emph{Inversion covariance.}
Suppose $0\notin\overline D$, set $\iota(x)=x/|x|^2$, and put
$D'=\iota(D)$. Then $D'$ is again a bounded connected Lipschitz domain,
and
\[
 \iota(K_D^{a,b})\overset{d}=
 K_{D'}^{\iota(a),\iota(b)}.
\]
For exit excursions the corresponding statement is
\[
 \iota(K_D^{a;f})\overset{d}=
 K_{D'}^{\iota(a);f^\iota},
 \qquad
 f^\iota(z)=\frac{f(\iota(z))}{|z|},
 \quad z\in\partial D'.
\]
Multiplying the exit weight by a positive constant does not change the
normalized law. For the conditional laws with a fixed exit point, this
endpoint factor cancels, giving
\[
 \iota(K_D^{a,z})\overset{d}=
 K_{D'}^{\iota(a),\iota(z)}
\]
for harmonic-measure-almost every $z$.

\smallskip
\noindent\emph{Justification.}
For interior endpoints, take one pair in \cref{prop:continuum-probe}.
Compact containment and recovery along
subsequences follow from 
\cref{lem:green-path-limit,prop:multipath-hittability,prop:common-decoration-stability}.
The stopped change of measure \eqref{eq:stopped-Doob-exact}, localized away
from the terminal point and combined with the reversal identity
\eqref{eq:LERW-reversal-law}, transfers the quasi-loop argument of
\cite[Section~6]{SS18}; the return estimates at the two fixed interior
endpoints allow us to remove the restrictions near those
endpoints. Thus these subsequential
limits are simple paths. The local-nonpolarity and Green function moment arguments of
\cref{lem:no-polar-subarc,lem:mixed-moments,prop:green-injective} apply with
the bridge range in place of the stopped Brownian range: the probes are
unchanged, and only the prescribed right-hand side of the recovery
transform changes. They give uniqueness and hence convergence of the full
sequence.

For exit excursions, direct decoration leaves the unique boundary point
unchanged. Weighting by $f$ therefore preserves the recovery property,
with the same weight on the Brownian exit point; continuity of $f$ also
passes this weighting through the joint discrete limit. The same
identification argument gives uniqueness, and disintegration gives the
almost-everywhere conditional statement.

Finally, the inversion identities of \cite[Lemmas~3.5 and~4.1]{JegoLupu26},
restricted to $D$, transport normalized bridges and the independent loop
soup. For a stopped Brownian path, the transformed law is instead the Doob
$h$-transform in $D'$, where $h(x)=|x|^{-1}$. Its density relative to
Brownian motion started at $\iota(a)$ and stopped on exiting $D'$ is
$h(W_{\tau_{D'}})/h(\iota(a))$. This gives exactly the exit weight
$f^\iota$ above. Inversion preserves intersections and closure, so
transporting the recovery identities and applying uniqueness proves the
covariance statements.
\end{remark}

\end{document}